\documentclass[a4paper,11pt]{amsart}
\usepackage{latexsym}
\usepackage{amssymb}
\usepackage{amsmath}

\newtheorem{theorem}{Theorem}
\newtheorem{proposition}[theorem]{Proposition}
\newtheorem{corollary}[theorem]{Corollary}
\newtheorem{lemma}[theorem]{Lemma}
\newtheorem{example}[theorem]{Example}

\newtheorem{definition}[theorem]{Definition}

\newcommand{\RMod}{R{\rm \mbox{-Mod}}}
\newcommand{\RPinj}{R{\rm \mbox{-Pinj}}}
\newcommand{\RInj}{R{\rm \mbox{-Inj}}}
\newcommand{\RFlat}{R{\rm \mbox{-Flat}}}
\newcommand{\RProj}{R{\rm \mbox{-Proj}}}

\newcommand{\RCotor}{R{\rm \mbox{-Cotor}}}

\newcommand{\mcEInj}{\mcE{\rm \mbox{-Inj}}}
\newcommand{\mcEinj}{\mcE{\rm \mbox{-inj}}}

\newcommand{\uMod}{\rm \underline{\mbox{Mod}}}

\newcommand{\Ab}{{\rm {\bf \mbox{Ab}}}}

\newcommand{\add}{{\rm \mbox{add}}}

\newcommand{\Arr}{{\rm \mbox{Arr}}}

\newcommand{\Ext}{{\rm \mbox{Ext}}}

\newcommand{\Hom}{{\rm \mbox{Hom}}}
\newcommand{\uHom}{{\underline {\rm \mbox{Hom}}}}
\newcommand{\im}{{\rm {\mbox{Im }}}}

\newcommand{\Ker}{{\rm \mbox{Ker }}}

\newcommand{\ME}{{\rm \mbox{ME}}}
\newcommand{\Mod}{{\rm \mbox{Mod}}}

\newcommand{\Ob}{{\rm \mbox{Ob}}}

\newcommand{\op}{{\rm \mbox{{\footnotesize op}}}}

\newcommand{\PO}{{\rm \mbox{PO}}}
\newcommand{\PE}{{\rm \mbox{PE}}}
\newcommand{\Pinj}{{\rm \mbox{Pinj}}}

\newcommand{\Tor}{{\rm \mbox{Tor}}}

\newcommand{\gra}{\alpha}

\newcommand{\grO}{\Omega}

\newcommand{\gro}{\omega}

\newcommand{\grf}{\varphi}
\newcommand{\grF}{\Phi}

\newcommand{\mcA}{\mathcal{A}}

\newcommand{\mcE}{\mathcal{E}}
\newcommand{\umcE}{\underline{\mathcal{E}}}
\newcommand{\mcF}{\mathcal{F}}
\newcommand{\mcC}{\mathcal{C}}
\newcommand{\mcI}{\mathcal{I}}
\newcommand{\mcJ}{\mathcal{J}}
\newcommand{\mcK}{\mathcal{K}}

\newcommand{\mcM}{\mathcal{M}}

\newcommand{\mcN}{\mathcal{N}}

\newcommand{\mcT}{\mathcal{T}}
\newcommand{\mcX}{\mathcal{X}}
\newcommand{\mcY}{\mathcal{Y}}

\newcommand{\bbA}{\mathbb{bA}}

\newcommand{\bbZ}{\mathbb{bZ}}

\newcommand{\isom}{\cong}

\newcommand{\dis}{\displaystyle}

\newcommand{\tlowername}[2]%
{$\stackrel{\makebox[1pt]{#1}}%
{\begin{picture}(0,0)%
\put(0,0){\makebox(0,6)[t]{\makebox[1pt]{$#2$}}}%
\end{picture}}$}%

\newcommand{\AR}[1]%
{\begin{picture}(#1,0)%
\put(0,0){\vector(1,0){#1}}%
\end{picture}}%

\newcommand{\DOTAR}[1]%
{\NUMBEROFDOTS=#1%
\divide\NUMBEROFDOTS by 3%
\begin{picture}(#1,0)%
\multiput(0,0)(3,0){\NUMBEROFDOTS}{\circle*{1}}%
\put(#1,0){\vector(1,0){0}}%
\end{picture}}%

\newcommand{\MONO}[1]%
{\begin{picture}(#1,0)%
\put(0,0){\vector(1,0){#1}}%
\put(2,-2){\line(0,1){4}}%
\end{picture}}%

\newcommand{\EPI}[1]%
{\begin{picture}(#1,0)(-#1,0)%
\put(-#1,0){\vector(1,0){#1}}%
\put(-6,-2){\line(0,1){4}}%
\end{picture}}%

\newcommand{\BIMO}[1]%
{\begin{picture}(#1,0)(-#1,0)%
\put(-#1,0){\vector(1,0){#1}}%
\put(-6,-2){\line(0,1){4}}%
\put(-#1,-2){\hspace{2pt}\line(0,1){4}}%
\end{picture}}%

\newcommand{\BIAR}[1]%
{\begin{picture}(#1,4)%
\put(0,0){\vector(1,0){#1}}%
\put(0,4){\vector(1,0){#1}}%
\end{picture}}%

\newcommand{\EQL}[1]%
{\begin{picture}(#1,0)%
\put(0,1){\line(1,0){#1}}%
\put(0,-1){\line(1,0){#1}}%
\end{picture}}%

\newcommand{\ADJAR}[1]%
{\begin{picture}(#1,4)%
\put(0,0){\vector(1,0){#1}}%
\put(#1,4){\vector(-1,0){#1}}%
\end{picture}}%

\newcommand{\BKAR}[1]%
{\begin{picture}(#1,0)%
\put(#1,0){\vector(-1,0){#1}}%
\end{picture}}%

\newcommand{\BKDOTAR}[1]%
{\NUMBEROFDOTS=#1%
\divide\NUMBEROFDOTS by 3%
\begin{picture}(#1,0)%
\multiput(#1,0)(-3,0){\NUMBEROFDOTS}{\circle*{1}}%
\put(0,0){\vector(-1,0){0}}%
\end{picture}}%

\newcommand{\BKMONO}[1]%
{\begin{picture}(#1,0)(-#1,0)%
\put(0,0){\vector(-1,0){#1}}%
\put(-2,-2){\line(0,1){4}}%
\end{picture}}%

\newcommand{\BKEPI}[1]%
{\begin{picture}(#1,0)%
\put(#1,0){\vector(-1,0){#1}}%
\put(6,-2){\line(0,1){4}}%
\end{picture}}%

\newcommand{\BKBIMO}[1]%
{\begin{picture}(#1,0)%
\put(#1,0){\vector(-1,0){#1}}%
\put(6,-2){\line(0,1){4}}%
\put(#1,-2){\hspace{-2pt}\line(0,1){4}}%
\end{picture}}%

\newcommand{\BKBIAR}[1]%
{\begin{picture}(#1,4)%
\put(#1,0){\vector(-1,0){#1}}%
\put(#1,4){\vector(-1,0){#1}}%
\end{picture}}%

\newcommand{\BKADJAR}[1]%
{\begin{picture}(#1,4)%
\put(0,4){\vector(1,0){#1}}%
\put(#1,0){\vector(-1,0){#1}}%
\end{picture}}%

\newcommand{\lowername}[2]%
{$\stackrel{\makebox[1pt]{#1}}%
{\begin{picture}(0,0)%
\truex{600}%
\put(0,0){\makebox(0,\value{x})[t]{\makebox[1pt]{$#2$}}}%
\end{picture}}$}%

\newcommand{\hcase}[2]%
{\makebox[0pt]%
{\raisebox{-1pt}[0pt][0pt]{#1{#2}}}}%

\newcommand{\Hcase}[3]%
{\makebox[0pt]
{\raisebox{-1pt}[0pt][0pt]%
{$\stackrel{\makebox[0pt]{$\textstyle{#2}$}}{#1{#3}}$}}}%

\newcommand{\hcasE}[3]%
{\makebox[0pt]%
{\raisebox{-9pt}[0pt][0pt]%
{\lowername{#1{#3}}{#2}}}}%

\newcommand{\hbicase}[2]%
{\makebox[0pt]%
{\raisebox{-2.5pt}[0pt][0pt]{#1{#2}}}}%

\newcommand{\Hbicase}[4]%
{\makebox[0pt]
{\raisebox{-10.5pt}[0pt][0pt]%
{$\stackrel{\makebox[0pt]{$\textstyle{#2}$}}%
{\mbox{\lowername{#1{#4}}{#3}}}$}}}%

\newcommand{\EAR}[1]%
{\begin{picture}(#1,0)%
\put(0,0){\vector(1,0){#1}}%
\end{picture}}%

\newcommand{\EDOTAR}[1]%
{\truex{100}\truey{300}%
\NUMBEROFDOTS=#1%
\divide\NUMBEROFDOTS by \value{y}%
\begin{picture}(#1,0)%
\multiput(0,0)(\value{y},0){\NUMBEROFDOTS}%
{\circle*{\value{x}}}%
\put(#1,0){\vector(1,0){0}}%
\end{picture}}%

\newcommand{\EMONO}[1]%
{\begin{picture}(#1,0)%
\put(0,0){\vector(1,0){#1}}%
\truex{300}\truey{600}%
\put(\value{x},-\value{x}){\line(0,1){\value{y}}}%
\end{picture}}%

\newcommand{\EEPI}[1]%
{\begin{picture}(#1,0)(-#1,0)%
\put(-#1,0){\vector(1,0){#1}}%
\truex{300}\truey{600}\truez{800}%
\put(-\value{z},-\value{x}){\line(0,1){\value{y}}}%
\end{picture}}%

\newcommand{\EBIMO}[1]%
{\begin{picture}(#1,0)(-#1,0)%
\put(-#1,0){\vector(1,0){#1}}%
\truex{300}\truey{600}\truez{800}%
\put(-\value{z},-\value{x}){\line(0,1){\value{y}}}%
\put(-#1,-\value{x}){\hspace{3pt}\line(0,1){\value{y}}}%
\end{picture}}%

\newcommand{\EBIAR}[1]%
{\truex{400}%
\begin{picture}(#1,\value{x})%
\put(0,0){\vector(1,0){#1}}%
\put(0,\value{x}){\vector(1,0){#1}}%
\end{picture}}%

\newcommand{\EEQL}[1]%
{\begin{picture}(#1,0)%
\truex{200}%
\put(0,\value{x}){\line(1,0){#1}}%
\put(0,0){\line(1,0){#1}}%
\end{picture}}%

\newcommand{\EADJAR}[1]%
{\truex{400}%
\begin{picture}(#1,\value{x})%
\put(0,0){\vector(1,0){#1}}%
\put(#1,\value{x}){\vector(-1,0){#1}}%
\end{picture}}%

\newcommand{\earv}[1]{\hcase{\EAR}{#100}}%

\newcommand{\ear}%
{\hspace{\SOURCE\unitlength}%
\hcase{\EAR}{\ARROWLENGTH}}%

\newcommand{\Earv}[2]{\Hcase{\EAR}{#1}{#200}}%

\newcommand{\Ear}[1]%
{\hspace{\SOURCE\unitlength}%
\Hcase{\EAR}{#1}{\ARROWLENGTH}}%

\newcommand{\eaR}[1]%
{\hspace{\SOURCE\unitlength}%
\hcasE{\EAR}{#1}{\ARROWLENGTH}}%

\newcommand{\edotar}%
{\hspace{\SOURCE\unitlength}%
\hcase{\EDOTAR}{\ARROWLENGTH}}%

\newcommand{\Edotar}[1]%
{\hspace{\SOURCE\unitlength}%
\Hcase{\EDOTAR}{#1}{\ARROWLENGTH}}%

\newcommand{\edotaR}[1]%
{\hspace{\SOURCE\unitlength}%
\hcasE{\EDOTAR}{#1}{\ARROWLENGTH}}%

\newcommand{\emono}%
{\hspace{\SOURCE\unitlength}%
\hcase{\EMONO}{\ARROWLENGTH}}%

\newcommand{\Emono}[1]%
{\hspace{\SOURCE\unitlength}%
\Hcase{\EMONO}{#1}{\ARROWLENGTH}}%

\newcommand{\emonO}[1]%
{\hspace{\SOURCE\unitlength}%
\hcasE{\EMONO}{#1}{\ARROWLENGTH}}%

\newcommand{\eepi}%
{\hspace{\SOURCE\unitlength}%
\hcase{\EEPI}{\ARROWLENGTH}}%

\newcommand{\Eepi}[1]%
{\hspace{\SOURCE\unitlength}%
\Hcase{\EEPI}{#1}{\ARROWLENGTH}}%

\newcommand{\eepI}[1]%
{\hspace{\SOURCE\unitlength}%
\hcasE{\EEPI}{#1}{\ARROWLENGTH}}%

\newcommand{\ebimo}%
{\hspace{\SOURCE\unitlength}%
\hcase{\EBIMO}{\ARROWLENGTH}}%

\newcommand{\Ebimo}[1]%
{\hspace{\SOURCE\unitlength}%
\Hcase{\EBIMO}{#1}{\ARROWLENGTH}}%

\newcommand{\ebimO}[1]%
{\hspace{\SOURCE\unitlength}%
\hcasE{\EBIMO}{#1}{\ARROWLENGTH}}%

\newcommand{\eiso}%
{\hspace{\SOURCE\unitlength}%
\Hcase{\EAR}{\cong}{\ARROWLENGTH}}%

\newcommand{\Eiso}[1]%
{\hspace{\SOURCE\unitlength}%
\Hcase{\EAR}{\cong#1}{\ARROWLENGTH}}%

\newcommand{\eisO}[1]%
{\hspace{\SOURCE\unitlength}%
\hcasE{\EAR}{\cong#1}{\ARROWLENGTH}}%

\newcommand{\ebiar}%
{\hspace{\SOURCE\unitlength}%
\hbicase{\EBIAR}{\ARROWLENGTH}}%

\newcommand{\Ebiar}[2]%
{\hspace{\SOURCE\unitlength}%
\Hbicase{\EBIAR}{#1}{#2}{\ARROWLENGTH}}%

\newcommand{\eeqlv}[1]{\hcase{\EEQL}{#100}}%

\newcommand{\eeql}%
{\hspace{\SOURCE\unitlength}%
\hbicase{\EEQL}{\ARROWLENGTH}}%

\newcommand{\eadjar}%
{\hspace{\SOURCE\unitlength}%
\hbicase{\EADJAR}{\ARROWLENGTH}}%

\newcommand{\Eadjar}[2]%
{\hspace{\SOURCE\unitlength}%
\Hbicase{\EADJAR}{#1}{#2}{\ARROWLENGTH}}%

\newcommand{\WAR}[1]%
{\begin{picture}(#1,0)%
\put(#1,0){\vector(-1,0){#1}}%
\end{picture}}%

\newcommand{\WDOTAR}[1]%
{\truex{100}\truey{300}%
\NUMBEROFDOTS=#1%
\divide\NUMBEROFDOTS by \value{y}%
\begin{picture}(#1,0)%
\multiput(#1,0)(-\value{y},0){\NUMBEROFDOTS}%
{\circle*{\value{x}}}%
\put(0,0){\vector(-1,0){0}}%
\end{picture}}%

\newcommand{\WMONO}[1]%
{\begin{picture}(#1,0)(-#1,0)%
\put(0,0){\vector(-1,0){#1}}%
\truex{300}\truey{600}%
\put(-\value{x},-\value{x}){\line(0,1){\value{y}}}%
\end{picture}}%

\newcommand{\WEPI}[1]%
{\begin{picture}(#1,0)%
\put(#1,0){\vector(-1,0){#1}}%
\truex{300}\truey{600}\truez{800}%
\put(\value{z},-\value{x}){\line(0,1){\value{y}}}%
\end{picture}}%

\newcommand{\WBIMO}[1]%
{\begin{picture}(#1,0)%
\put(#1,0){\vector(-1,0){#1}}%
\truex{300}\truey{600}\truez{800}%
\put(\value{z},-\value{x}){\line(0,1){\value{y}}}%
\put(#1,-\value{x}){\hspace{-3pt}\line(0,1){\value{y}}}%
\end{picture}}%

\newcommand{\WBIAR}[1]%
{\truex{400}%
\begin{picture}(#1,\value{x})%
\put(#1,0){\vector(-1,0){#1}}%
\put(#1,\value{x}){\vector(-1,0){#1}}%
\end{picture}}%

\newcommand{\WADJAR}[1]%
{\truex{400}%
\begin{picture}(#1,\value{x})%
\put(0,\value{x}){\vector(1,0){#1}}%
\put(#1,0){\vector(-1,0){#1}}%
\end{picture}}%

\newcommand{\war}%
{\hspace{\SOURCE\unitlength}%
\hcase{\WAR}{\ARROWLENGTH}}%

\newcommand{\War}[1]%
{\hspace{\SOURCE\unitlength}%
\Hcase{\WAR}{#1}{\ARROWLENGTH}}%

\newcommand{\waR}[1]%
{\hspace{\SOURCE\unitlength}%
\hcasE{\WAR}{#1}{\ARROWLENGTH}}%

\newcommand{\wdotar}%
{\hspace{\SOURCE\unitlength}%
\hcase{\WDOTAR}{\ARROWLENGTH}}%

\newcommand{\Wdotar}[1]%
{\hspace{\SOURCE\unitlength}%
\Hcase{\WDOTAR}{#1}{\ARROWLENGTH}}%

\newcommand{\wdotaR}[1]%
{\hspace{\SOURCE\unitlength}%
\hcasE{\WDOTAR}{#1}{\ARROWLENGTH}}%

\newcommand{\wmono}%
{\hspace{\SOURCE\unitlength}%
\hcase{\WMONO}{\ARROWLENGTH}}%

\newcommand{\Wmono}[1]%
{\hspace{\SOURCE\unitlength}%
\Hcase{\WMONO}{#1}{\ARROWLENGTH}}%

\newcommand{\wmonO}[1]%
{\hspace{\SOURCE\unitlength}%
\hcasE{\WMONO}{#1}{\ARROWLENGTH}}%

\newcommand{\wepi}%
{\hspace{\SOURCE\unitlength}%
\hcase{\WEPI}{\ARROWLENGTH}}%

\newcommand{\Wepi}[1]%
{\hspace{\SOURCE\unitlength}%
\Hcase{\WEPI}{#1}{\ARROWLENGTH}}%

\newcommand{\wepI}[1]%
{\hspace{\SOURCE\unitlength}%
\hcasE{\WEPI}{#1}{\ARROWLENGTH}}%

\newcommand{\wbimo}%
{\hspace{\SOURCE\unitlength}%
\hcase{\WBIMO}{\ARROWLENGTH}}%

\newcommand{\Wbimo}[1]%
{\hspace{\SOURCE\unitlength}%
\Hcase{\WBIMO}{#1}{\ARROWLENGTH}}%

\newcommand{\wbimO}[1]%
{\hspace{\SOURCE\unitlength}%
\hcasE{\WBIMO}{#1}{\ARROWLENGTH}}%

\newcommand{\wiso}%
{\hspace{\SOURCE\unitlength}%
\Hcase{\WAR}{\cong}{\ARROWLENGTH}}%

\newcommand{\Wiso}[1]%
{\hspace{\SOURCE\unitlength}%
\Hcase{\WAR}{#1}{\ARROWLENGTH}}%

\newcommand{\wisO}[1]%
{\hspace{\SOURCE\unitlength}%
\hcasE{\WAR}{#1}{\ARROWLENGTH}}%

\newcommand{\wbiar}%
{\hspace{\SOURCE\unitlength}%
\hbicase{\WBIAR}{\ARROWLENGTH}}%

\newcommand{\Wbiar}[2]%
{\hspace{\SOURCE\unitlength}%
\Hbicase{\WBIAR}{#1}{#2}{\ARROWLENGTH}}%

\newcommand{\weql}%
{\hspace{\SOURCE\unitlength}%
\hbicase{\EEQL}{\ARROWLENGTH}}%

\newcommand{\wadjar}%
{\hspace{\SOURCE\unitlength}%
\hbicase{\WADJAR}{\ARROWLENGTH}}%

\newcommand{\Wadjar}[2]%
{\hspace{\SOURCE\unitlength}%
\Hbicase{\WADJAR}{#1}{#2}{\ARROWLENGTH}}%

\newcommand{\vcase}[2]{#1{#2}}%

\newcommand{\Vcase}[3]{\makebox[0pt]%
{\makebox[0pt][r]{\raisebox{0pt}[0pt][0pt]{${#2}\hspace{2pt}$}}}#1{#3}}%

\newcommand{\vcasE}[3]{\makebox[0pt]%
{#1{#3}\makebox[0pt][l]{\raisebox{0pt}[0pt][0pt]{\hspace{2pt}$#2$}}}}%

\newcommand{\vbicase}[2]{\makebox[0pt]{{#1{#2}}}}%

\newcommand{\Vbicase}[4]{\makebox[0pt]%
{\makebox[0pt][r]{\raisebox{0pt}[0pt][0pt]{$#2$\hspace{4pt}}}#1{#4}%
\makebox[0pt][l]{\raisebox{0pt}[0pt][0pt]{\hspace{5pt}$#3$}}}}%

\newcommand{\SAR}[1]%
{\begin{picture}(0,0)%
\put(0,0){\makebox(0,0)%
{\begin{picture}(0,#1)%
\put(0,#1){\vector(0,-1){#1}}%
\end{picture}}}\end{picture}}%

\newcommand{\SDOTAR}[1]%
{\truex{100}\truey{300}%
\NUMBEROFDOTS=#1%
\divide\NUMBEROFDOTS by \value{y}%
\begin{picture}(0,0)%
\put(0,0){\makebox(0,0)%
{\begin{picture}(0,#1)%
\multiput(0,#1)(0,-\value{y}){\NUMBEROFDOTS}%
{\circle*{\value{x}}}%
\put(0,0){\vector(0,-1){0}}%
\end{picture}}}\end{picture}}%

\newcommand{\SMONO}[1]%
{\begin{picture}(0,0)%
\put(0,0){\makebox(0,0)%
{\begin{picture}(0,#1)%
\put(0,#1){\vector(0,-1){#1}}%
\truex{300}\truey{600}%
\put(0,#1){\begin{picture}(0,0)%
\put(-\value{x},-\value{x}){\line(1,0){\value{y}}}\end{picture}}%
\end{picture}}}\end{picture}}%

\newcommand{\SEPI}[1]%
{\begin{picture}(0,0)%
\put(0,0){\makebox(0,0)%
{\begin{picture}(0,#1)%
\put(0,#1){\vector(0,-1){#1}}%
\truex{300}\truey{600}\truez{800}%
\put(-\value{x},\value{z}){\line(1,0){\value{y}}}%
\end{picture}}}\end{picture}}%

\newcommand{\SBIMO}[1]%
{\begin{picture}(0,0)%
\put(0,0){\makebox(0,0)%
{\begin{picture}(0,#1)%
\put(0,#1){\vector(0,-1){#1}}%
\truex{300}\truey{600}\truez{800}%
\put(0,#1){\begin{picture}(0,0)%
\put(-\value{x},-\value{x}){\line(1,0){\value{y}}}\end{picture}}%
\put(-\value{x},\value{z}){\line(1,0){\value{y}}}%
\end{picture}}}\end{picture}}%

\newcommand{\SBIAR}[1]%
{\begin{picture}(0,0)%
\truex{200}%
\put(0,0){\makebox(0,0)%
{\begin{picture}(0,#1)\put(-\value{x},#1){\vector(0,-1){#1}}%
\put(\value{x},#1){\vector(0,-1){#1}}%
\end{picture}}}\end{picture}}%

\newcommand{\SEQL}[1]%
{\begin{picture}(0,0)%
\truex{100}%
\put(0,0){\makebox(0,0)%
{\begin{picture}(0,#1)\put(-\value{x},#1){\line(0,-1){#1}}%
\put(\value{x},#1){\line(0,-1){#1}}%
\end{picture}}}\end{picture}}%

\newcommand{\sarv}[1]{\vcase{\SAR}{#100}}%

\newcommand{\sar}{\sarv{50}}%

\newcommand{\Sarv}[2]{\Vcase{\SAR}{#1}{#200}}%

\newcommand{\Sar}[1]{\Sarv{#1}{50}}%

\newcommand{\saRv}[2]{\vcasE{\SAR}{#1}{#200}}%

\newcommand{\saR}[1]{\saRv{#1}{50}}%

\newcommand{\Sisov}[2]%
{\Vbicase{\SAR}{#1\hspace{-2pt}}{\hspace{-2pt}\cong}{#200}}%

\newcommand{\seqlv}[1]{\vbicase{\SEQL}{#100}}%

\newcommand{\seql}{\seqlv{50}}%

\newcommand{\NAR}[1]%
{\begin{picture}(0,0)%
\put(0,0){\makebox(0,0)%
{\begin{picture}(0,#1)\put(0,0){\vector(0,1){#1}}%
\end{picture}}}\end{picture}}%

\newcommand{\NDOTAR}[1]%
{\truex{100}\truey{300}%
\NUMBEROFDOTS=#1%
\divide\NUMBEROFDOTS by \value{y}%
\begin{picture}(0,0)%
\put(0,0){\makebox(0,0)%
{\begin{picture}(0,#1)%
\multiput(0,0)(0,\value{y}){\NUMBEROFDOTS}%
{\circle*{\value{x}}}%
\put(0,#1){\vector(0,1){0}}%
\end{picture}}}\end{picture}}%

\newcommand{\NMONO}[1]%
{\begin{picture}(0,0)%
\put(0,0){\makebox(0,0)%
{\begin{picture}(0,#1)%
\put(0,0){\vector(0,1){#1}}%
\truex{300}\truey{600}%
\put(-\value{x},\value{x}){\line(1,0){\value{y}}}%
\end{picture}}}%
\end{picture}}%

\newcommand{\NEPI}[1]%
{\begin{picture}(0,0)%
\put(0,0){\makebox(0,0)%
{\begin{picture}(0,#1)%
\put(0,0){\vector(0,1){#1}}%
\truex{300}\truey{600}\truez{800}%
\put(0,#1){\begin{picture}(0,0)%
\put(-\value{x},-\value{z}){\line(1,0){\value{y}}}\end{picture}}%
\end{picture}}}\end{picture}}%

\newcommand{\NBIMO}[1]%
{\begin{picture}(0,0)%
\put(0,0){\makebox(0,0)%
{\begin{picture}(0,#1)%
\put(0,0){\vector(0,1){#1}}%
\truex{300}\truey{600}\truez{800}%
\put(-\value{x},\value{x}){\line(1,0){\value{y}}}%
\put(0,#1){\begin{picture}(0,0)%
\put(-\value{x},-\value{z}){\line(1,0){\value{y}}}\end{picture}}%
\end{picture}}}\end{picture}}%

\newcommand{\NBIAR}[1]%
{\begin{picture}(0,0)%
\truex{200}%
\put(0,0){\makebox(0,0)%
{\begin{picture}(0,#1)\put(-\value{x},0){\vector(0,1){#1}}%
\put(\value{x},0){\vector(0,1){#1}}%
\end{picture}}}\end{picture}}%

\newcommand{\Nisov}[2]%
{\Vbicase{\NAR}{#1\hspace{-2pt}}{\hspace{-2pt}\cong}{#200}}%

\newcommand{\fdcase}[3]{\begin{picture}(0,0)%
\put(0,-150){#1}%
\truex{200}\truey{600}\truez{600}%
\put(-\value{x},-\value{x}){\makebox(0,\value{z})[r]{${#2}$}}%
\put(\value{x},-\value{y}){\makebox(0,\value{z})[l]{${#3}$}}%
\end{picture}}%

\newcommand{\fdbicase}[3]{\begin{picture}(0,0)%
\put(0,-150){#1}%
\truex{800}\truey{50}%
\put(-\value{x},\value{y}){${#2}$}%
\truex{200}\truey{950}%
\put(\value{x},-\value{y}){${#3}$}%
\end{picture}}%

\newcommand{\NEDOTAR}%
{\truex{100}\truey{212}%
\NUMBEROFDOTS=5800%
\divide\NUMBEROFDOTS by \value{y}%
\begin{picture}(0,0)%
\multiput(-2900,-2900)(\value{y},\value{y}){\NUMBEROFDOTS}%
{\circle*{\value{x}}}%
\put(2900,2900){\vector(1,1){0}}%
\end{picture}}%

\newcommand{\NEEQL}{\begin{picture}(0,0)%
\put(-2900,-2900){\begin{picture}(0,0)%
\truex{70}%
\put(-\value{x},\value{x}){\line(1,1){5800}}%
\put(\value{x},-\value{x}){\line(1,1){5800}}%
\end{picture}}\end{picture}}%

\newcommand{\SWAR}{\begin{picture}(0,0)%
\put(2900,2900){\vector(-1,-1){5800}}%
\end{picture}}%

\newcommand{\SWDOTAR}%
{\truex{100}\truey{212}%
\NUMBEROFDOTS=5800%
\divide\NUMBEROFDOTS by \value{y}%
\begin{picture}(0,0)%
\multiput(2900,2900)(-\value{y},-\value{y}){\NUMBEROFDOTS}%
{\circle*{\value{x}}}%
\put(-2900,-2900){\vector(-1,-1){0}}%
\end{picture}}%

\newcommand{\SWARV}[1]{\begin{picture}(0,0)%
\put(0,0){\makebox(0,0){\begin{picture}(#1,#1)%
\put(#1,#1){\vector(-1,-1){#1}}\end{picture}}}%
\end{picture}}%

\newcommand{\Swar}[1]{\fdcase{\SWAR}{#1}{}}%

\newcommand{\swaR}[1]{\fdcase{\SWAR}{}{#1}}%

\newcommand{\swdotar}{\fdcase{\SWDOTAR}{}{}}%

\newcommand{\sweql}{\fdbicase{\NEEQL}{}{}}%

\newcommand{\swarv}[1]{\fdcase{\SWARV{#100}}{}{}}%

\newcommand{\Swarv}[2]{\fdcase{\SWARV{#200}}{#1}{}}%

\newcommand{\swaRv}[2]{\fdcase{\SWARV{#200}}{}{#1}}%

\newcommand{\sdcase}[3]{\begin{picture}(0,0)%
\put(0,-150){#1}%
\truex{100}\truez{600}%
\put(\value{x},\value{x}){\makebox(0,\value{z})[l]{${#2}$}}%
\truex{300}\truey{800}%
\put(-\value{x},-\value{y}){\makebox(0,\value{z})[r]{${#3}$}}%
\end{picture}}%

\newcommand{\SEAR}{\begin{picture}(0,0)%
\put(-2900,2900){\vector(1,-1){5800}}%
\end{picture}}%

\newcommand{\SEDOTAR}%
{\truex{100}\truey{212}%
\NUMBEROFDOTS=5800%
\divide\NUMBEROFDOTS by \value{y}%
\begin{picture}(0,0)%
\multiput(-2900,2900)(\value{y},-\value{y}){\NUMBEROFDOTS}%
{\circle*{\value{x}}}%
\put(2900,-2900){\vector(1,-1){0}}%
\end{picture}}%

\newcommand{\SEARV}[1]{\begin{picture}(0,0)%
\put(0,0){\makebox(0,0){\begin{picture}(#1,#1)%
\put(0,#1){\vector(1,-1){#1}}\end{picture}}}%
\end{picture}}%

\newcommand{\Sear}[1]{\sdcase{\SEAR}{#1}{}}%

\newcommand{\seaR}[1]{\sdcase{\SEAR}{}{#1}}%

\newcommand{\Searv}[2]{\sdcase{\SEARV{#200}}{#1}{}}%

\newcommand{\seaRv}[2]{\sdcase{\SEARV{#200}}{}{#1}}%

\newcommand{\NWDOTAR}%
{\truex{100}\truey{212}%
\NUMBEROFDOTS=5800%
\divide\NUMBEROFDOTS by \value{y}%
\begin{picture}(0,0)%
\multiput(2900,-2900)(-\value{y},\value{y}){\NUMBEROFDOTS}%
{\circle*{\value{x}}}%
\put(-2900,2900){\vector(-1,1){0}}%
\end{picture}}%

\newcommand{\ENEAR}[2]%
{\makebox[0pt]{\begin{picture}(0,0)%
\put(0,-150){\makebox(0,0){\begin{picture}(0,0)%
\put(-6600,-3300){\vector(2,1){13200}}%
\truex{200}\truey{800}\truez{600}%
\put(-\value{x},\value{x}){\makebox(0,\value{z})[r]{${#1}$}}%
\put(\value{x},-\value{y}){\makebox(0,\value{z})[l]{${#2}$}}%
\end{picture}}}\end{picture}}}%

\newcommand{\ESEAR}[2]%
{\makebox[0pt]{\begin{picture}(0,0)%
\put(0,-150){\makebox(0,0){\begin{picture}(0,0)%
\put(-6600,3300){\vector(2,-1){13200}}%
\truex{200}\truey{800}\truez{600}%
\put(\value{x},\value{x}){\makebox(0,\value{z})[l]{${#1}$}}%
\put(-\value{x},-\value{y}){\makebox(0,\value{z})[r]{${#2}$}}%
\end{picture}}}\end{picture}}}%

\newcommand{\WNWAR}[2]%
{\makebox[0pt]{\begin{picture}(0,0)%
\put(0,-150){\makebox(0,0){\begin{picture}(0,0)%
\put(6600,-3300){\vector(-2,1){13200}}%
\truex{200}\truey{800}\truez{600}%
\put(\value{x},\value{x}){\makebox(0,\value{z})[l]{${#1}$}}%
\put(-\value{x},-\value{y}){\makebox(0,\value{z})[r]{${#2}$}}%
\end{picture}}}\end{picture}}}%

\newcommand{\WSWAR}[2]%
{\makebox[0pt]{\begin{picture}(0,0)%
\put(0,-150){\makebox(0,0){\begin{picture}(0,0)%
\put(6600,3300){\vector(-2,-1){13200}}%
\truex{200}\truey{800}\truez{600}%
\put(-\value{x},\value{x}){\makebox(0,\value{z})[r]{${#1}$}}%
\put(\value{x},-\value{y}){\makebox(0,\value{z})[l]{${#2}$}}%
\end{picture}}}\end{picture}}}%

\newcommand{\NNEAR}[2]%
{\raisebox{-1pt}[0pt][0pt]{\begin{picture}(0,0)%
\put(0,0){\makebox(0,0){\begin{picture}(0,0)%
\put(-3300,-6600){\vector(1,2){6600}}%
\truex{100}\truez{600}%
\put(-\value{x},\value{x}){\makebox(0,\value{z})[r]{${#1}$}}%
\put(\value{x},-\value{z}){\makebox(0,\value{z})[l]{${#2}$}}%
\end{picture}}}\end{picture}}}%

\newcommand{\SSWAR}[2]%
{\raisebox{-1pt}[0pt][0pt]{\begin{picture}(0,0)%
\put(0,0){\makebox(0,0){\begin{picture}(0,0)%
\put(3300,6600){\vector(-1,-2){6600}}%
\truex{100}\truez{600}%
\put(-\value{x},\value{x}){\makebox(0,\value{z})[r]{${#1}$}}%
\put(\value{x},-\value{z}){\makebox(0,\value{z})[l]{${#2}$}}%
\end{picture}}}\end{picture}}}%

\newcommand{\SSEAR}[2]%
{\raisebox{-1pt}[0pt][0pt]{\begin{picture}(0,0)%
\put(0,0){\makebox(0,0){\begin{picture}(0,0)%
\put(-3300,6600){\vector(1,-2){6600}}%
\truex{200}\truez{600}%
\put(\value{x},\value{x}){\makebox(0,\value{z})[l]{${#1}$}}%
\put(-\value{x},-\value{z}){\makebox(0,\value{z})[r]{${#2}$}}%
\end{picture}}}\end{picture}}}%

\newcommand{\NNWAR}[2]%
{\raisebox{-1pt}[0pt][0pt]{\begin{picture}(0,0)%
\put(0,0){\makebox(0,0){\begin{picture}(0,0)%
\put(3300,-6600){\vector(-1,2){6600}}%
\truex{200}\truez{600}%
\put(\value{x},\value{x}){\makebox(0,\value{z})[l]{${#1}$}}%
\put(-\value{x},-\value{z}){\makebox(0,\value{z})[r]{${#2}$}}%
\end{picture}}}\end{picture}}}%

\newcommand{\Necurve}[2]%
{\begin{picture}(0,0)%
\truex{1300}\truey{2000}\truez{200}%
\put(0,\value{x}){\oval(#200,\value{y})[t]}%
\put(0,\value{x}){\makebox(0,0){\begin{picture}(#200,0)%
\put(#200,0){\vector(0,-1){\value{z}}}%
\put(0,0){\line(0,-1){\value{z}}}\end{picture}}}%
\truex{2500}%
\put(0,\value{x}){\makebox(0,0)[b]{${#1}$}}%
\end{picture}}%

\newcommand{\Nwcurve}[2]%
{\begin{picture}(0,0)%
\truex{1300}\truey{2000}\truez{200}%
\put(0,\value{x}){\oval(#200,\value{y})[t]}%
\put(0,\value{x}){\makebox(0,0){\begin{picture}(#200,0)%
\put(#200,0){\line(0,-1){\value{z}}}%
\put(0,0){\vector(0,-1){\value{z}}}\end{picture}}}%
\truex{2500}%
\put(0,\value{x}){\makebox(0,0)[b]{${#1}$}}%
\end{picture}}%

\newcommand{\Securve}[2]%
{\begin{picture}(0,0)%
\truex{1300}\truey{2000}\truez{200}%
\put(0,-\value{x}){\oval(#200,\value{y})[b]}%
\put(0,-\value{x}){\makebox(0,0){\begin{picture}(#200,0)%
\put(#200,0){\vector(0,1){\value{z}}}%
\put(0,0){\line(0,1){\value{z}}}\end{picture}}}%
\truex{2500}%
\put(0,-\value{x}){\makebox(0,0)[t]{${#1}$}}%
\end{picture}}%

\newcommand{\Swcurve}[2]%
{\begin{picture}(0,0)%
\truex{1300}\truey{2000}\truez{200}%
\put(0,-\value{x}){\oval(#200,\value{y})[b]}%
\put(0,-\value{x}){\makebox(0,0){\begin{picture}(#200,0)%
\put(#200,0){\line(0,1){\value{z}}}%
\put(0,0){\vector(0,1){\value{z}}}\end{picture}}}%
\truex{2500}%
\put(0,-\value{x}){\makebox(0,0)[t]{${#1}$}}%
\end{picture}}%

\newcommand{\Escurve}[2]%
{\begin{picture}(0,0)%
\truex{1400}\truey{2000}\truez{200}%
\put(\value{x},0){\oval(\value{y},#200)[r]}%
\put(\value{x},0){\makebox(0,0){\begin{picture}(0,#200)%
\put(0,0){\vector(-1,0){\value{z}}}%
\put(0,#200){\line(-1,0){\value{z}}}\end{picture}}}%
\truex{2500}%
\put(\value{x},0){\makebox(0,0)[l]{${#1}$}}%
\end{picture}}%

\newcommand{\Encurve}[2]%
{\begin{picture}(0,0)%
\truex{1400}\truey{2000}\truez{200}%
\put(\value{x},0){\oval(\value{y},#200)[r]}%
\put(\value{x},0){\makebox(0,0){\begin{picture}(0,#200)%
\put(0,0){\line(-1,0){\value{z}}}%
\put(0,#200){\vector(-1,0){\value{z}}}\end{picture}}}%
\truex{2500}%
\put(\value{x},0){\makebox(0,0)[l]{${#1}$}}%
\end{picture}}%

\newcommand{\Wscurve}[2]%
{\begin{picture}(0,0)%
\truex{1300}\truey{2000}\truez{200}%
\put(-\value{x},0){\oval(\value{y},#200)[l]}%
\put(-\value{x},0){\makebox(0,0){\begin{picture}(0,#200)%
\put(0,0){\vector(1,0){\value{z}}}%
\put(0,#200){\line(1,0){\value{z}}}\end{picture}}}%
\truex{2400}%
\put(-\value{x},0){\makebox(0,0)[r]{${#1}$}}%
\end{picture}}%

\newcommand{\Wncurve}[2]%
{\begin{picture}(0,0)%
\truex{1300}\truey{2000}\truez{200}%
\put(-\value{x},0){\oval(\value{y},#200)[l]}%
\put(-\value{x},0){\makebox(0,0){\begin{picture}(0,#200)%
\put(0,0){\line(1,0){\value{z}}}%
\put(0,#200){\vector(1,0){\value{z}}}\end{picture}}}%
\truex{2400}%
\put(-\value{x},0){\makebox(0,0)[r]{${#1}$}}%
\end{picture}}%

\newcount\SCALE%

\newcount\NUMBER%

\newcount\LINE%

\newcount\COLUMN%

\newcount\WIDTH%

\newcount\SOURCE%

\newcount\ARROW%

\newcount\TARGET%

\newcount\ARROWLENGTH%

\newcount\NUMBEROFDOTS%

\newcounter{x}%

\newcounter{y}%

\newcounter{z}%

\newcounter{horizontal}%

\newcounter{vertical}%

\newskip\itemlength%

\newskip\firstitem%

\newskip\seconditem%

\newcommand{\printarrow}{}%

\newcommand{\truex}[1]{%
\NUMBER=#1%
\multiply\NUMBER by 100%
\divide\NUMBER by \SCALE%
\setcounter{x}{\NUMBER}}%

\newcommand{\truey}[1]{%
\NUMBER=#1%
\multiply\NUMBER by 100%
\divide\NUMBER by \SCALE%
\setcounter{y}{\NUMBER}}%

\newcommand{\truez}[1]{%
\NUMBER=#1%
\multiply\NUMBER by 100%
\divide\NUMBER by \SCALE%
\setcounter{z}{\NUMBER}}%

\newcommand{\changecounters}[1]{%
\SOURCE=\ARROW%
\ARROW=\TARGET%
\settowidth{\itemlength}{#1}%
\ifdim \itemlength > 2800\unitlength%
\addtolength{\itemlength}{-2800\unitlength}%
\TARGET=\itemlength%
\divide\TARGET by 1310%
\multiply\TARGET by 100%
\divide\TARGET by \SCALE%
\else%
\TARGET=0%
\fi%
\ARROWLENGTH=5000%
\advance\ARROWLENGTH by -\SOURCE%
\advance\ARROWLENGTH by -\TARGET%
\advance\SOURCE by -\TARGET}%

\newcommand{\initialize}[1]{%
\LINE=0%
\COLUMN=0%
\WIDTH=0%
\ARROW=0%
\TARGET=0%
\changecounters{#1}%
\renewcommand{\printarrow}{#1}%
\begin{center}%
\vspace{10pt}%
\begin{picture}(0,0)}%

\newcommand{\DIAG}[1]{%
\SCALE=100%
\setlength{\unitlength}{655sp}%
\initialize{\mbox{$#1$}}}%

\newcommand{\DIAGV}[2]{%
\SCALE=#1%
\setlength{\unitlength}{655sp}%
\multiply\unitlength by \SCALE%
\divide\unitlength by 100%
\initialize{\mbox{$#2$}}}%

\newcommand{\n}[1]{%
\changecounters{\mbox{$#1$}}%
\put(\COLUMN,\LINE){\makebox(0,0){\printarrow}}%
\thinlines%
\renewcommand{\printarrow}{\mbox{$#1$}}%
\advance\COLUMN by 4000}%

\newcommand{\nn}[1]{%
\put(\COLUMN,\LINE){\makebox(0,0){\printarrow}}%
\thinlines%
\ifnum \WIDTH < \COLUMN%
\WIDTH=\COLUMN%
\else%
\fi%
\advance\LINE by -4000%
\COLUMN=0%
\ARROW=0%
\TARGET=0%
\changecounters{\mbox{$#1$}}%
\renewcommand{\printarrow}{\mbox{$#1$}}}%

\newcommand{\conclude}{%
\put(\COLUMN,\LINE){\makebox(0,0){\printarrow}}%
\thinlines%
\ifnum \WIDTH < \COLUMN%
\WIDTH=\COLUMN%
\else%
\fi%
\setcounter{horizontal}{\WIDTH}%
\setcounter{vertical}{-\LINE}%
\end{picture}}%

\newcommand{\diag}{%
\conclude%
\raisebox{0pt}[0pt][\value{vertical}\unitlength]{}%
\hspace*{\value{horizontal}\unitlength}%
\vspace{10pt}%
\end{center}%
\setlength{\unitlength}{1pt}}%

\newcommand{\diagv}[3]{%
\conclude%
\NUMBER=#1%
\rule{0pt}{\NUMBER pt}%
\hspace*{-#2pt}%
\raisebox{0pt}[0pt][\value{vertical}\unitlength]{}%
\hspace*{\value{horizontal}\unitlength}
\NUMBER=#3%
\advance\NUMBER by 10%
\vspace*{\NUMBER pt}%
\end{center}%
\setlength{\unitlength}{1pt}}%

\newcommand{\N}[1]%
{\raisebox{0pt}[7pt][0pt]{$#1$}}%

\newcommand{\movename}[3]{%
\hspace{#2pt}%
\raisebox{#3pt}[5pt][2pt]{\raisebox{#3pt}{$#1$}}%
\hspace{-#2pt}}%

\newcommand{\crosslength}[2]{%
\settowidth{\firstitem}{#1}%
\settowidth{\seconditem}{#2}%
\ifdim\firstitem < \seconditem%
\itemlength=\seconditem%
\else%
\itemlength=\firstitem%
\fi%
\divide\itemlength by 2%
\hspace{\itemlength}}%

\newcommand{\cross}[2]{%
\crosslength{\mbox{$#1$}}{\mbox{$#2$}}%
\begin{picture}(0,0)%
\put(0,0){\makebox(0,0){$#1$}}%
\thinlines%
\put(0,0){\makebox(0,0){$#2$}}%
\thinlines%
\end{picture}%
\crosslength{\mbox{$#1$}}{\mbox{$#2$}}}%

\begin{document}

\footskip30pt

\date{}

\title{Powers of the Phantom Ideal}

\author{X.H.\ Fu}
\address{School of Mathematics and Statistics, Northeast Normal University, Changchun, China}
\email{fuxianhui@gmail.com}

\author{I. Herzog}
\address{The Ohio State University at Lima, Lima, Ohio, USA}
\email{herzog.23@osu.edu}
\thanks{The first author was partially supported by the National Natural Science Foundation of China, Grant No. 11301062; the second by NSF Grant DMS 12-01523.}



\begin{abstract}
It is proved that if $G$ is a finite group, then the order of $G$ is a proper upper bound for the phantom number of $G.$ This answers a question of Benson and Gnacadja. More specifically, if $k$ is a field whose characteristic divides the order of $G,$ and $\grF$ is the ideal of phantom morphisms in the stable category $k[G]$-$\uMod$ of modules over the group algebra $k[G],$ then $\grF^{n-1} = 0,$ where $n$ is the nilpotency index of the Jacobson radical $J$ of $k[G].$ If $R$ is a semiprimary ring, with $J^n =0,$ and $\grF$ denotes the phantom ideal in the module category $\RMod,$ then $\grF^n$ is the ideal of morphisms that factor through a projective module. If $R$ is a right coherent ring and every cotorsion left $R$-module has a coresolution of length $n$ by pure injective modules, then $\Phi^{n+1}$ is the ideal of morphisms that factor through a flat module. 
        
These results are obtained by introducing the mono-epi (ME) exact structure on the morphisms of an exact category $(\mcA; \mcE).$ This exact structure $(\Arr (\mcA); \ME)$ is used to develop further the ideal approximation theory of $(\mcA; \mcE)$, by proving new versions of Salce's Lemma, the Ghost Lemma of Christensen, and Wakamatsu's Lemma. Salce's Lemma states that if $(\mcA; \mcE)$ has enough injective morphisms and projective morphisms, then the map $\mcI \mapsto \mcI^{\perp}$ on ideals is a bijective correspondence between the class of special precovering ideals of $(\mcA; \mcE)$ and that of its special preenveloping ideals. The exact category $(\Arr (\mcA); \ME)$ of morphisms allows us to introduce the notion of an extension $i \star j$ of morphisms in an exact category $(\mcA; \mcE),$ and the notion of an extension $\mcI \diamond \mcJ$ of ideals of $\mcA.$ The Ghost Lemma, instrumental in proving the consequences above, asserts that the class of special precovering (resp., special preenveloping) ideals is closed under products and extensions and that the bijective correspondence of Salce's Lemma satisfies $(\mcI \mcJ)^{\perp} = \mcJ^{\perp} \diamond \mcI^{\perp}$ and $(\mcI \diamond \mcJ)^{\perp} = \mcJ^{\perp} \mcI^{\perp}.$ Wakamatsu's Lemma is the statement that if a covering ideal $\mcI$ is closed under extensions $\mcI \diamond \mcI = \mcI,$ then $\mcI$ is a special precovering ideal that possesses a syzygy ideal $\grO (\mcI) \subseteq \mcI^{\perp}$ generated by objects. 
\end{abstract}

\maketitle
\pagestyle{plain}

\section{Introduction}

Let $\mcT$ be a triangulated category and $\mcT^c$ the subcategory of compact objects (see~\cite{N}). A morphism $f : X \to Y$ in $\mcT$ is a phantom morphism~\cite[Def 2.4]{N1} if for every morphism $c : C \to X,$ with $C \in \mcT^c,$ the diagram
\DIAGV{80}
{C} \n {\Ear{c}} \n {X} \nn
{} \n {\seaR{0}} \n {\saR{f}} \nn
{} \n {} \n {Y}
\diag
is commutative. The first examples of phantom morphisms arose in algebraic topology (see~\cite{McG}), in the work of Adams and Walker~\cite{AW}, with $\mcT$ the category of homotopy spectra. In the representation theory of groups, Benson and Gnacadja~\cite{BG2} discovered examples of phantom morphisms when $\mcT = k[G]$-$\uMod$ is the stable category of modules over the group algebra $k[G],$ where $k$ is a field.

For the triangulated category $\mcT = k[G]$-$\uMod,$ phantom morphisms were first investigated by Gnacadja~\cite{Gn}. A morphism 
$f : X \to Y$ in $k[G]$-$\Mod$ is a phantom morphism, when considered as a morphism in $k[G]$-$\uMod,$ if for every finitely presented left $k[G]$-module $C,$ the composition $fc$ factors through some projective module $P,$
\DIAGV{80}
{C} \n {\Ear{c}} \n {X} \nn
{\sar} \n {} \n {\saR{f}} \nn
{P} \n {\ear} \n {Y.}
\diag
The second author~\cite{H} considered the same condition on a morphism $f : X \to Y$ in the category $\RMod$ of left modules over an associative ring $R$ with identity. This is equivalent~\cite[Prop 36]{FGHT} to the condition that the induced natural transformation $\Tor_1^R (-,f) : \Tor_1^R (-,X) \to \Tor_1^R (-,Y)$ vanishes, so that in the context of $R$-modules a phantom morphism is the morphism version of a flat module.

Phantom morphisms constitute an ideal, denoted by $\grF,$ in both a triangulated category $\mcT$ and the module category $\RMod.$ Neeman~\cite{N1} was the first to consider phantom morphisms in a general setting, introducing conditions sufficient for the triangulated category $\mcT$ to be phantomless, $\grF = 0.$ For the category of homotopy spectra, as well as more general triangulated categories satisfying Brown Representability, Christensen and Strickland~\cite[Thm 1.2]{ChS} and Neeman~\cite[Cor 4.4]{N2} proved that $\grF^2 = 0.$ Recently, Muro and Raventos~\cite[Cor 6.26]{MR} showed that if the subcategory of compact objects is replaced by the (more general) notion of the category of $\gra$-compact objects, $\gra$ a regular cardinal, then the ideal $\grF_{\gra}$ of $\gra$-phantoms satisfies $(\cap_{n < \gro}\; \grF_{\gra}^n)^2 = 0.$ Benson~\cite{Ben}, however, found a class of group algebras for which $\grF^2 \neq 0.$ 

Benson and Gnacadja~\cite{BG1} noted that if the pure global dimension of the category $k[G]$-$\Mod$ is bounded by $n,$ then $\grF^{n+1} = 0$ in the stable category $k[G]$-$\uMod.$ In most cases, it is possible to artifically boost the pure global dimension of a group algebra $k[G]$ by increasing the cardinality of $k,$ but their work suggests that there exists a finite bound, the {\em phantom number} of $G,$ for the nilpotency index of the phantom ideal $\grF$ in $k[G]$-$\uMod$ for every field $k.$ This is confirmed by the theory developed in this article as follows. Recall that a ring $R$ is {\em semiprimary} if the Jacobson radical $J = J(R)$ is nilpotent and $R/J$ is semisimple artinian.  \bigskip

\noindent{\bf Theorem~\ref{semiprimary}.} {\it If $R$ is a semiprimary ring with $J^n = 0,$ then $\grF^n = \langle \RProj \rangle$ in the module category $\RMod.$} \bigskip

The proof follows the strategy used by Chebolu, Christensen and Min\'{a}\v{c}~\cite{CCM} to obtain a similar bound for the ghost number of a finite $p$-group. If $M$ is a left $R$-module, then the Loewy series $\{ J^iM \}_{i \leq n}$ is a filtration of $M,$ of length at most $n,$ whose factors are semisimple, hence pure injective. One then develops a theory of special precovering ideals in an exact category (in this case $\RMod$) to prove an analogue (Theorem~\ref{the ghost lemma}) of the Ghost Lemma~\cite[Thm 1.1]{Chr}. This version of the Ghost Lemma implies that every $R$-module $M$ that can be filtered by a series of length $n$ whose factors are pure injective is right $\Ext$-orthogonal to $\grF^n.$

For the special case of a Quasi-Frobenius ring~\cite{NY}, this bound on the nilpotency index of the phantom ideal in the stable category $R$-$\uMod$ is lowered by $1,$ because every module decomposes as a direct sum $M = E \oplus M'$ where $E$ is projective/injective and the Loewy length of $M'$ is bounded by $n-1,$ where $n$ is the nilpotency index of $J.$ \bigskip

\noindent{\bf Theorem~\ref{QF}.} {\it If $R$ is a QF ring with nonzero Jacobson radical $J,$ then $J^n = 0$ implies that $\grF^{n-1} = 0$ in the stable category $R$-$\uMod.$} \bigskip

For example, if $G = \bbZ/2 \times \bbZ/2$ is the Klein $4$-group, and the characteristic of $k$ is $2,$ then $J^3 = 0$ in $k[G],$ so that Theorem~\ref{QF} implies that $\grF^2 = 0$ in the stable category $k[G]$-$\uMod,$ a result established by Benson and Gnacadja~\cite[\S 4.6]{BG1} when $k$ is countable. On the other hand, it is a consequence of the Pure Semisimple Conjecture for QF rings~\cite[Cor 5.3]{H-pps} that a QF ring is phantomless if and only if it is of finite representation type~\cite[Prop 41]{FGHT}. Because the group algebra $k[\bbZ/2 \times \bbZ/2]$ is not of finite representation type~\cite[Thm 4.4.4]{B book}, $\grF \neq 0$ in the stable category. Theorem~\ref{QF} leads to the following positive resolution of a problem~\cite[Question 5.2.3]{BG1} posed by Benson and Gnacadja. \bigskip

\noindent{\bf Corollary~\ref{group bound}.} {\it Let $G$ be a finite group and $k$ a field. If $\grF$ denotes the ideal of phantom morphisms in the stable category $k[G]$-$\uMod$ of modules over the group algebra $k[G],$ then $\grF^{|G| - 1} = 0.$} \bigskip

A proof of Theorem~\ref{QF} can be obtained by applying a dualized form of Christensen's Ghost Lemma to the {\em injective} class 
(see~\cite[\S 2]{Chr}) $(\grF, \mbox{R-}$\underline{\Pinj}$)$ in the triangulated category $k[G]$-$\uMod$ of stable $k[G]$-modules, but Theorem~\ref{semiprimary} covers all artin algebras, and therefore every finite-dimensional algebra, as well as every finite ring. For the class of coherent rings, we build on the work of Xu~\cite{X} to attain the following related criterion, which also improves the bound provided by the left pure global dimension of $R.$ \bigskip

\noindent{\bf Corollary~\ref{Benson-G}.} Let $R$ be a right coherent ring such that every cotorsion left $R$-module $C$ has a coresolution
\DIAGV{60}
{0} \n {\ear} \n {C} \n {\ear} \n {I_0} \n {\ear} \n {I_1} \n {\ear} \n {\cdots} \n {\ear} \n {I_n} \n {\ear} \n {0}
\diag
with each $I_k$ pure injective. Then $\grF^{n+1} = \langle \RFlat \rangle.$

\bigskip

The relationship between phantom morphisms and the theory of purity had already been noted by Christensen and Strickland~\cite{ChS}. For the derived category $D(R)$ of a ring $R$, this was made more precise by Beligiannis~\cite{Bel} and Christensen, Keller and Neeman~\cite{CKN}, who used the difference between the pure global dimension of $R$ and its homological dimension to construct examples where Brown Representability fails. Indeed, the construction by Gray and McGibbon~\cite{GrM} of a phantom preenvelope in the category of homotopy spectra is the suspension of something analogous to a pure syzygy of a module. In a compactly generated triangulated category, H.\ Krause~\cite[Thm D]{Kra} proved the existence of phantom precovers by a dual argument, considering the desuspension of something analogous to the pure cosyzygy of a module. Employing an argument reminiscent of triangle constructions in the stable category of modules over a group algebra $k[G],$ the second author~\cite[Prop 6]{H} proved the existence of phantom morphisms in the module category $\RMod:$ given a left $R$-module $M,$ let $p : R^{(\gra)} \to M$ be an epimorphism from a free $R$-module, and take the pushout along the pure injective envelope $e : K \to \PE (K)$ of the syzygy $K = \grO (M),$
\DIAGV{80}
{0} \n {\ear} \n {K} \n {\ear} \n {R^{(\gra )}} \n {\Ear{p}} \n {M} \n {\ear} \n {0} \nn 
{} \n {} \n {\saR{e}} \n {} \n {\sar} \n {} \n {\seql} \nn 
{0} \n {\ear} \n {\PE (K)} \n {\ear} \n {F} \n {\Ear{\grf}} \n {M} \n {\ear} \n {0.} 
\diag   
The morphism $\grf : F \to M$ is then a phantom precover. This simple construction stands in stark contrast to the technically involved proofs, due to Bican, El Bashir, and Enochs~\cite{BEE} (see also~\cite{ET, ElB}), of the existence of flat precovers in a module category. 

Based on this construction of a phantom precover, Guil Asensio, Torrecillas and the authors formulated a theory~\cite{FGHT} of ideal approximations in the setting of an exact category $(\mcA; \mcE)$ (see~\cite{B, GR}). This theory generalizes to ideals of morphisms the classical theory of approximations, i.e., precovers and preenvelopes, for subcategories of objects, pioneered by Auslander and Smal\o~\cite[Ch VII]{Aus} and Enochs~\cite{En} (see~\cite{BR, EJ, GT, X}). An ideal $\mcI$ of $\mcA$ is {\em precovering} if for every object $A$ in $\mcA$ there exists a deflation $i : I \to A$ in $\mcI$ such that every morphism $i' : I' \to A$ in $\mcI$ factors through $i,$
\DIAGV{70}
{} \n {} \n {I'} \nn
{} \n {\swdotar} \n {\saR{i'}} \nn
{I} \n {\Ear{i}} \n {A.}
\diag   
Ideal Approximation Theory~\cite{FGHT, H2, O, EGO} for exact categories is devoted to the study of precovering ideals, and the dual notion of preenveloping ideals, with emphasis on the notion of a special precovering (resp., special preenveloping) ideal. A {\em special} $\mcI$-precover of an object $A$ is a deflation $i_1 : I_1 \to A$ that occurs in a conflation $\Xi$ arising as a pushout 
\DIAGV{70}
{\Xi' :} \n {K_0} \n {\ear} \n {I_0} \n {\ear} \n {A} \nn
{} \n {\saR{k}} \n {} \n {\sar} \n {} \n {\seql} \nn
{\Xi :} \n {K_1} \n {\ear} \n {I_1} \n {\Ear{i_1}} \n {A} 
\diag
along a morphism $k \in \mcI^{\perp}.$ An ideal $\mcI$ is {\em special precovering} (resp., {\em special preenveloping}) if every object has a special $\mcI$-precover (resp., {\em special $\mcI$-preenvelope}).

In this article, we develop Ideal Approximation Theory further by introducing an exact structure on the category $\Arr (\mcA)$ of morphisms (arrows) of an exact category $(\mcA; \mcE).$ The category $\Arr (\mcA)$ has the natural exact structure whose conflations are the morphisms of conflations in $(\mcA; \mcE).$ This exact category, which we denote by $(\Arr (\mcA); \Arr (\mcE)),$ has been studied by Estrada, Guil Asensio and \"{O}zbek~\cite{EGO}, who observe its shortcomings in their Remark 3.4. In Definition~\ref{mono epi def}, we introduce the notion of a mono-epi (ME) morphism of conflations and denote by $\ME \subseteq \Arr (\mcE)$ the collection of such morphisms of conflations. \bigskip

\noindent {\bf Theorem~\ref{mono-epi}.} {\it The mono-epi substructure $(\Arr (\mcA); \ME) \subseteq (\Arr (\mcA); \Arr (\mcE))$ is exact.} \bigskip 

We use the exact structure on $(\Arr (\mcA); \ME)$ to find a place within Ideal Approximation Theory for three of the pillars of the classical theory: Salce's Lemma~\cite{S}, the Ghost Lemma~\cite{Chr} and Wakamatsu's Lemma~\cite{W}. An ideal version of Salce's Lemma was already proved in~\cite{FGHT} as the implication ($2$) $\Rightarrow$ ($3$) of Theorem 1. The hypotheses are weakened here to obtain the following. \bigskip

\noindent {\bf Theorem~\ref{Salce}.} {\rm (Salce's Lemma)} {\it Let $(\mcA; \mcE)$ be an exact category with enough injective morphisms and enough projective morphisms. The rule $\mcI \mapsto \mcI^{\perp}$ is a bijective correspondence between the class of special precovering ideals $\mcI$ of $(\mcA; \mcE)$ and that of its special preenveloping ideals $\mcK.$ The inverse rule is given by $\mcK \mapsto {^{\perp}}\mcK.$}  
\bigskip

\noindent Just as the classical Salce's Lemma gives rise to the central notion of a complete cotorsion pair, Theorem~\ref{Salce} leads to the notion of a complete ideal cotorsion pair $(\mcI, \mcI^{\perp}),$ where $\mcI$ is a special precovering ideal. The exact structure on $(\Arr (\mcA); \ME)$ allows us to introduce the concept of an extension $i \star j$ of morphisms and, if $\mcI$ and $\mcJ$ are ideals of $\mcA,$ the concept of an extension of ideals $\mcI \diamond \mcJ = \langle i \star j \; | \; i \in \mcI, j \in \mcJ \rangle.$  \bigskip

\noindent {\bf Theorem~\ref{the ghost lemma}.} {\rm (The Ghost Lemma)} {\it Let $(\mcA; \mcE)$ be an exact category with enough injective morphisms and enough projective morphisms. The class of special precovering (resp., preenveloping) ideals is closed under products $\mcI \mcJ$ and extensions $\mcI \diamond \mcJ.$ Moreover, the bijective correspondence $\mcI \mapsto \mcI^{\perp}$ satisfies} 
$$(\mcI \mcJ)^{\perp} = \mcJ^{\perp} \diamond \mcI^{\perp} \;\; \mbox{and} \;\; (\mcI \diamond \mcJ)^{\perp} = \mcJ^{\perp} \mcI^{\perp}.$$ \smallskip

Theorem~\ref{the ghost lemma} is an analogue of Christensen's Ghost Lemma~\cite[Thm 1.1]{Chr}, which serves as the model for several results that play a key role in the respective theories of dimensions of triangulated categories~\cite[Lemma 4.11]{R}, representation dimensions of artin algebras~(\cite{R2, BIKO} and~\cite[Lemma 2.1]{B}), and strongly finitely generated triangulated categories~\cite[Thm 4]{OS}. If $\mcI$ and $\mcJ$ are object ideals, then so is the extension ideal $\mcI \diamond \mcJ$ (Theorem~\ref{object ext}). Theorem~\ref{the ghost lemma} implies that if $(\mcI, \mcI^{\perp})$ and $(\mcJ, \mcJ^{\perp})$ are complete ideal cotorsion pairs such that $\mcI^{\perp}$ and $\mcJ^{\perp}$ are object ideals, then so is
$(\mcI \mcJ, \mcJ^{\perp} \diamond \mcI^{\perp}).$ Such complete ideal cotorsion pairs are the analogues in the present context of the projective classes studied by Christensen. 

In many arguments, we can avoid the hypothesis that there exist enough injective or projective morphisms, by working directly with the syzygy morphism of a special precover or, for a special precovering ideal $\mcI,$ with an ideal $\grO (\mcI) \subseteq \mcI^{\perp}$ generated by syzygy morphisms. For example, we generalize Christensen's Ghost Lemma for projective classes by calling a special precovering ideal $\mcI$ object-special (Definition~\ref{ospi} and Proposition~\ref{object syzygy}) if {\em some} syzygy ideal of $\grO (\mcI)$ is an object ideal, and proving that such ideals are closed under products (Corollary 23). The ability to do this seems to be a virtue of Ideal Approximation Theory, formally expressed by Theorem~\ref{syzygy ideal} and Proposition~\ref{ext by objects}, that is absent in the classical theory. Another example of this phenomenon is the Chain Rule for syzygies (Theorem~\ref{syzygy ext}).

According to Theorem~\ref{the ghost lemma}, the bijective correspondence $\mcI \mapsto \mcI^{\perp}$ of Salce's Lemma associates an idempotent special precovering ideal to a special preenveloping ideal closed under extensions, and conversely. In the last section of the paper, we take up these two classes of ideals, but under the hypothesis that they be covering, rather than special precovering. \bigskip

\noindent {\bf Theorem~\ref{Wakamatsu}.} {\rm (Wakamatsu's Lemma)} {\it Every covering ideal $\mcI,$ closed under extensions, is an object-special precovering ideal.} \bigskip
 
\noindent {\em Acknowledgements.} A debt is owing to Ron Gentle, Mark Hovey and Jan \v{S}\v{t}ov\'{i}\v{c}ek for the insightful advice they gave us on preliminary presentations of our results. We also thank David Benson, who encouraged us to study the phantom number. Part of this project was carried out while the first author was a visitor at The Ohio State University at Lima and while both authors were vistors at Nanjing University; we gratefully acknowledge the hospitality of both institutions.

\section{Preliminaries}

Let $\mcA$ be an additive category. An {\em ideal} $\mcI$ of $\mcA$ is an additive subbifunctor of the additive bifunctor 
$\Hom (-,-) : \mcA^{\op} \times \mcA \to \Ab,$ where $\Ab$ denotes the category of abelian groups: to every pair $(A,B)$ of objects in $\mcA,$ the subfunctor $\mcI$ associates a subgroup $\mcI (A,B) \subseteq \Hom (A,B)$ such that if $g : A \to B$ belongs to $\mcI (A,B),$ then the composition $fgh : X \to Y$ belongs to $\mcI (X,Y),$ whenever $f : B \to Y$ and $h : X \to A$ are morphisms in $\mcA.$ 

If $\mcM$ is a class of morphisms in $\mcA,$ then $\langle \mcM \rangle$ denotes the smallest ideal of $\mcA$ that contains $\mcM.$ For example, the {\em product} of two ideals $\mcI$ and $\mcJ$ of $\mcA$ is given by
$$\mcI \mcJ := \langle ij \; | \; i \in \mcI \; \mbox{and} \; j \in \mcJ \; \mbox{are composable} \; \rangle.$$
A subcategory $\mcX$ of $\mcA$ generates the ideal $\langle \mcX \rangle := \langle 1_X \; | \; X \in \mcX \rangle,$ and an ideal $\mcI$ of $\mcA$ is called an {\em object ideal} if it is of the form $\mcI = \langle \mcX \rangle$ for some subcategory $\mcX$ of $\mcA.$ Equivalently, 
$\mcI = \langle \Ob (\mcI) \rangle.$ 

An {\em additive} subcategory $\mcC \subseteq \mcA$ is a subcategory that is closed under finite direct sums and direct summands. The {\em additive closure} of a subcategory $\mcX$ of $\mcA$ is the smallest additive subcategory $\add (\mcX)$ that contains $\mcX.$ If $\mcX$ is closed under finite direct sums, then an object of $\mcA$ belongs to $\add (\mcX)$ provided it is a summand of some object of $\mcX.$

\begin{proposition} \label{ideal/additive}
Given an ideal $\mcI$ of $\mcA,$ the subcategory $\Ob (\mcI)$ of $\mcA$ is additive. Given a subcategory $\mcC$ closed under finite direct sums, the object ideal $\langle \mcC \rangle$ consists of the morphisms $f : A \to B$ in $\mcA$ that factor as $f : A \to C \to B$ through some object $C$ in $\mcC.$ The rule $\mcC \mapsto \langle \mcC \rangle$ is a bijective correspondence between the class of additive subcategories $\mcC$ of $\mcA$ and the object ideals of $\mcA;$ the inverse rule is given by $\mcI \mapsto \Ob (\mcI).$
\end{proposition}

\begin{proof}
To see that $\Ob (\mcI)$ is closed under direct summands, suppose that $A \oplus B$ belongs to $\Ob (\mcI).$ Let $\iota_A : A \to A \oplus B$ (resp., $\pi_A : A \oplus B \to A$) be the structural injection (resp., projection) associated to the summand $A.$ Then $1_A = \pi_A 1_{A \oplus B} \iota_A$ also belongs to $\mcI.$ To see that $\Ob (\mcI)$ is closed under finite direct sums, note that $1_{A \oplus B} = \iota_A 1_A \pi_A + \iota_B 1_B \pi_B.$

Let $\mcC$ be a subcategory of $\mcA$ that is closed under finite direct sums. A morphism that factors through some object of $\mcC$ clearly belongs to $\langle \mcC \rangle.$ Conversely, every morphism in $\langle \mcC \rangle$ is of the form $\sum_i a_i1_{C_i}b_i$ and therefore factors through the finite direct sum $\oplus_i C_i \in \mcC.$ 

To prove that the given correspondence is bijective, recall that if $\mcI$ is an object ideal, then $\langle \Ob (\mcI) \rangle = \mcI,$ whereas,
if $\mcC$ is an additive subcategory, then an object $A$ belongs to $\Ob (\langle \mcC \rangle)$ if and only if $1_A$ factors through an object 
$C \in \mcC.$ But then $A$ is a direct summand of $C$ and so too belongs to $\mcC.$
\end{proof}

The ideas of the proof of Proposition~\ref{ideal/additive} may also be used to infer the following.

\begin{proposition} \label{add closure}
If $\mcX$ is a subcategory of $\mcA,$ then $\add (\mcX) = \Ob (\langle \mcX \rangle).$
\end{proposition}

In this paper, we rely heavily on the theory of exact categories. We closely follow B\"{u}hler's comprehensive treatment~\cite{B} as the standard reference, but we use the terminology of Keller~\cite{GR}. An {\em exact structure} $(\mcA; \mcE)$ on an additive category $\mcA$ consists of a collection $\mcE$ of distinguished kernel-cokernel pairs
\DIAGV{80}
{\Xi :} \n {A} \n {\Ear{m}} \n {B} \n {\Ear{e}} \n {C}
\diag 
called {\em conflations.} The morphism $m$ is called the {\em inflation} of $\Xi;$ the morphism $e$ the {\em deflation.} More generally, a morphism $m$ (resp., $e$) is called an {\em inflation} (resp., {\em deflation}) if it is the inflation (resp., deflation) of some conflation in $\mcE.$ The collection $\mcE$ is closed under isomorphism and satisfies the axioms:
\begin{description}
\item[E$_0$] for every object $A \in \mcA,$ the morphism $1_A$ is an inflation;
\item[E$_0^{\op}$] for every object $A \in \mcA,$ the morphism $1_A$ is a deflation;
\item[E$_1$] inflations are closed under composition;
\item[E$_1^{\op}$] deflations are closed under composition;
\item[E$_2$] the pushout of an inflation along an arbitrary morphism exists and yields an inflation;
\item[E$_2^{\op}$] the pullback of a deflation along an arbitrary morphism exists and yields a deflation.  
\end{description}
\bigskip

The {\em arrow} category $\Arr (\mcA)$ of a category $\mcA$ is the category whose objects $a : A_0 \to A_1$ are the morphisms (arrows) of $\mcA,$ and a morphism $f : a \to b$ in $\Arr (\mcA)$ is given by a pair of morphisms $f = (f_0, f_1)$ of $\mcA$ for which the diagram
\DIAGV{80}
{A_0} \n {\Ear{f_0}} \n {B_0} \nn
{\saR{a}} \n {} \n {\saR{b}} \nn
{A_1} \n {\Ear{f_1}} \n {B_1}
\diag
commutes. We will adhere to the convention that in a $2$-dimensional diagram, arrows will be depicted vertically, as above, while morphisms are depicted horizontally. In a $3$-dimensional diagram, arrows will appear orthogonal to the page, while morphisms of arrows will appear to be inside the page. There is a full and faithful functor $\mcA \to \Arr (\mcA)$ given by $A \mapsto 1_A : A \to A.$ An arrow $a : A_0 \to A_1$ is isomorphic to an object $1_A$ of $\mcA$ if and only if it is an isomorphism. In that case, $a \isom 1_{A_0} \isom 1_{A_1}.$

If $(\mcA; \mcE)$ is an exact category, then it is readily verified that $(\Arr (\mcA), \Arr (\mcE))$ satisfies the axioms for an exact 
category~\cite[Cor 2.10]{B}, where a kernel-cokernel pair 
\DIAGV{80}
{\xi :} \n {a} \n {\Ear{f}} \n {b} \n {\Ear{g}} \n {c}
\diag
of $\Arr (\mcA)$ belongs to $\Arr (\mcE)$ provided that it is a morphism
\DIAGV{80}
{\Xi_0 :} \n {A_0} \n {\Ear{f_0}} \n {B_0} \n {\Ear{g_0}} \n {C_0} \nn
{} \n {\saR{a}} \n {} \n {\saR{b}} \n {} \n {\saR{c}} \nn
{\Xi_1 :} \n {A_1} \n {\Ear{f_1}} \n {B_1} \n {\Ear{g_1}} \n {C_1}
\diag
of conflations in $(\mcA; \mcE).$ The full and faithful embedding $(\mcA; \mcE) \subseteq (\Arr (\mcA); \Arr (\mcE))$ is exact, in the sense that if $\Xi : A \to B \to C$ is a conflation in $\mcE,$ then $1_{\Xi} : 1_A \to 1_B \to 1_C$ is one in $\Arr (\mcE).$ 

\section{The Mono-Epi Exact Structure of Arrows}

Let $(\mcA; \mcE)$ be an exact category. If a conflation $\xi : i \to a \to j$ in $\Arr (\mcE)$ is considered as a morphism of conflations in 
$(\mcA; \mcE),$ then it has a {\em pullback-pushout} factorization~\cite[Proposition 3.1]{B}
\DIAGV{80}
{\Xi_0:} \n {I_0} \n {\ear} \n {A_0} \n {\ear} \n {J_0} \nn
{} \n {\saR{i}} \n {} \n {\sar} \n {} \n {\seql} \nn
{\Xi' :} \n {I_1} \n {\ear} \n {A'} \n {\ear} \n {J_0} \nn
{} \n {\seql} \n {} \n {\sar} \n {} \n {\saR{j}} \nn
{\Xi_1:} \n {I_1} \n {\ear} \n {A_1} \n {\ear} \n {J_1,}
\diag
where $\Xi'$ is a conflation in $(\mcA; \mcE).$ Looking at this factorization, we see that $\xi$ is null-homotopic if and only if the conflation $\Xi'$ in $(\mcA; \mcE)$ is split. To motivate the next definition, let us recall that the category $\umcE$ whose objects are the conflations of 
$(\mcA; \mcE),$ and whose morphisms are the morphisms of $(\Arr (\mcA); \Arr (\mcE))$ modulo split exact conflations, is an abelian category~\cite{F}. If a conflation $\xi$ in $(\Arr (\mcA); \Arr (\mcE))$ is considered as a morphism in $\umcE,$ then the pullback-pushout factorization of $\xi$ is just the epi-mono factorization obtained from the abelian structure of $\umcE.$ 

\begin{definition} \label{mono epi def} A conflation $\xi : i \to a \to j$ in $\Arr (\mcE)$ is called {\em mono-epi (ME)} if there is a factorization
\DIAGV{80}
{\Xi_0:} \n {I_0} \n {\ear} \n {A_0} \n {\Ear{e_0}} \n {J_0} \nn
{} \n {\seql} \n {} \n {\Sar{a^1}} \n {} \n {\saR{j}} \nn
{\Xi:} \n {I_0} \n {\Ear{m}} \n {A} \n {\Ear{e}} \n {J_1} \nn
{} \n {\saR{i}} \n {} \n {\Sar{a^2}} \n {} \n {\seql} \nn
{\Xi_1:} \n {I_1} \n {\Ear{m_1}} \n {A_1} \n {\ear} \n {J_1,}
\diag
of $\xi,$ where the middle row is a conflation in $(\mcA; \mcE).$ Denote by $\ME \subseteq \Arr (\mcE)$ the collection of mono-epi conflations in 
$\Arr (\mcE).$ 
\end{definition}

An $\ME$-inflation is therefore a monomorphism $m : i \to a$ of arrows for which the arrow $a$ admits a factorization $a = a^2 a^1$ so that the morphism $m : I_0 \to A$ in the commutative diagram
\DIAGV{80}
{I_0} \n {\Ear{m_0}} \n {A_0} \nn
{\seql} \n {} \n {\saR{a^1}} \nn
{I_0} \n {\Ear{m}} \n {A} \nn
{\saR{i}} \n {\PO} \n {\saR{a^2}} \nn
{I_1} \n {\Ear{m_1}} \n {A_1}
\diag
is an inflation in $(\mcA; \mcE),$ and the bottom square is a pushout diagram.

\begin{theorem} \label{mono-epi}
The mono-epi substructure $(\Arr (\mcA); \ME) \subseteq (\Arr (\mcA); \Arr (\mcE))$ is exact. 
\end{theorem}

\begin{proof} 
Let us verify Axioms E$_0,$ E$_1,$ and E$_2^{\op}$ of an exact structure for $(\Arr (\mcA); \ME),$ the verification of the other axioms being dual.
Axiom E$_0$ that for every arrow $a \in \Arr (\mcA),$ the identity morphism $1_a : a \to a$ is an $ME$-inflation is easy to verify by the characterization above of an ME-inflation. To verify Axiom E$_1,$ which asserts that the composition of two ME-inflations is again such, consider such a composition $i \stackrel{m}{\to} a \stackrel{n}{\to} b$ as depicted by the diagram
\DIAGV{80}
{I_0} \n {\Ear{m_0}} \n {A_0} \n {\Ear{n_0}} \n {B_0} \nn
{\seql} \n {} \n {\seql} \n {} \n {\saR{b^1}}  \nn
{I_0} \n {\ear} \n {A_0} \n {\Ear{n'}} \n {B}  \nn
{\seql} \n {} \n {\saR{a^1}} \n {\PO} \n {}  \nn
{I_0} \n {\Ear{m'}} \n {A} \n {} \n {\saRv{b^2}{120}}  \nn
{\Sar{i}} \n {\PO} \n {\saR{a^2}} \n {} \n {}  \nn
{I_1} \n {\Ear{m_1}} \n {A_1} \n {\Ear{n_1}} \n {B_1,} 
\diag
all of whose horizontal maps are inflations in $(\mcA; \mcE).$ The pushout in the lower right rectangle may be factored by taking the pushout of $n'$ and $a^1$ to obtain the diagram
\DIAGV{80}
{I_0} \n {\Ear{m_0}} \n {A_0} \n {\Ear{n_0}} \n {B_0} \nn
{\seql} \n {} \n {\seql} \n {} \n {\saR{b^1}}  \nn
{I_0} \n {\ear} \n {A_0} \n {\Ear{n'}} \n {B}  \nn
{\seql} \n {} \n {\saR{a^1}} \n {\PO} \n {\saR{b_1^2}}  \nn
{I_0} \n {\Ear{m'}} \n {A} \n {\Ear{n''}} \n {B'}  \nn
{\Sar{i}} \n {\PO} \n {\saR{a^2}} \n {\PO} \n {\saR{b_2^2}}  \nn
{I_1} \n {\Ear{m_1}} \n {A_1} \n {\Ear{n_1}} \n {B_1,} 
\diag
which coarsens to
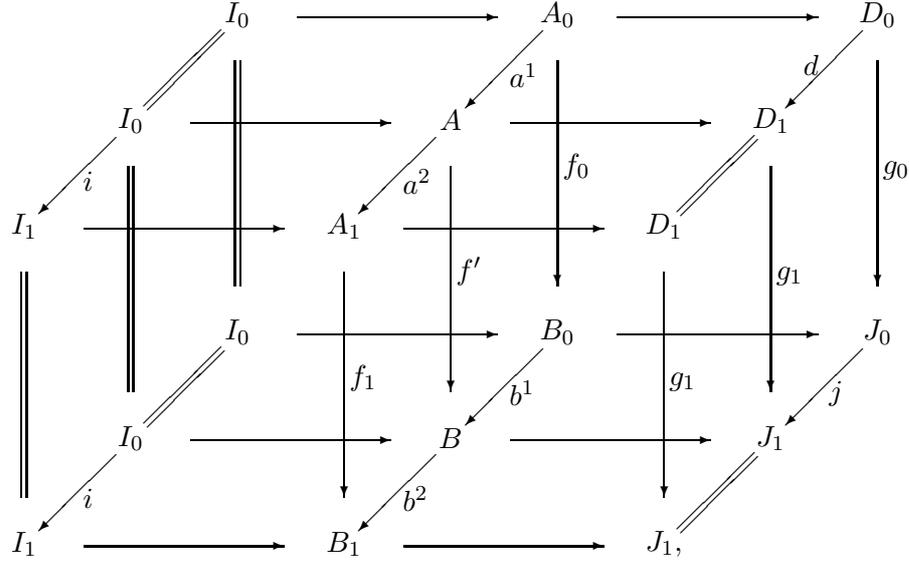
\DIAGV{80}
{I_0} \n {\Ear{m_0}} \n {A_0} \n {\Ear{n_0}} \n {B_0} \nn
{\seql} \n {} \n {} \n {} \n {\saR{b_1^2 b^1}}  \nn
{I_0} \n {\Ear{m'}} \n {A} \n {\Ear{n''}} \n {B'}  \nn
{\saR{i}} \n {\PO} \n {} \n {} \n {\saR{b_2^2}}  \nn
{I_1} \n {\Ear{m_1}} \n {A_1} \n {\Ear{n_1}} \n {B_1,} 
\diag
as required.

To verify Axiom E$_2^{\op}$ for an exact category, suppose that an ME-conflation $\xi : i \to b \to j$ is given and let $g : d \to j$ be an arbitrary morphism in $\Arr (\mcA).$ The pullback along $g$ with respect to the mono-epi factorization of $\xi$ is obtained by taking the pullbacks in $(\mcA; \mcE)$ along the vertical morphisms depicted in the diagram
\DIAGV{40}
{} \n {} \n {} \n {} \n {} \n {} \n {} \n {} \n {} \n {} \n {} \n {} \n {} \n {} \n {} \n {} \n {D_0} \nn 
{} \n {} \n {} \n {} \n {} \n {} \n {} \n {} \n {} \n {} \n {} \n {} \n {} \n {} \n {} \n {\Swar{d}} \n {} \nn 
{} \n {} \n {} \n {} \n {} \n {} \n {} \n {} \n {} \n {} \n {} \n {} \n {} \n {} \n {D_1} \n {} \n {} \nn 
{} \n {} \n {} \n {} \n {} \n {} \n {} \n {} \n {} \n {} \n {} \n {} \n {} \n {\sweql} \n {} \n {} \n {\saRv{g_0}{170}} \nn 
{} \n {} \n {} \n {} \n {} \n {} \n {} \n {} \n {} \n {} \n {} \n {} \n {D_1} \n {} \n {} \n {} \n {} \nn   
{} \n {} \n {} \n {} \n {} \n {} \n {} \n {} \n {} \n {} \n {} \n {} \n {} \n {} \n {\saRv{g_1}{170}} \n {} \n {} \nn 
{} \n {} \n {} \n {} \n {I_0} \n {} \n {} \n {\earv{150}} \n {} \n {} \n {B_0} \n {} \n {} \n {\earv{150}} \n {} \n {} \n {J_0} \nn 
{} \n {} \n {} \n {\sweql} \n {} \n {} \n {} \n {} \n {} \n {\swaR{b^1}} \n {} \n {} \n {\saRv{g_1}{170}} \n {} \n {} \n {\swaR{j}} \nn 
{} \n {} \n {I_0} \n {} \n {} \n {\earv{150}} \n {} \n {} \n {B} \n {} \n {} \n {\earv{150}} \n {} \n {} \n {J_1} \nn 
{} \n {\swaR{i}} \n {} \n {} \n {} \n {} \n {} \n {\swaR{b^2}} \n {} \n {} \n {} \n {} \n {} \n {\sweql} \nn 
{I_1} \n {} \n {} \n {\earv{150}} \n {} \n {} \n {B_1} \n {} \n {} \n {\earv{150}} \n {} \n {} \n {J_1.}  
\diag
When these pullbacks are taken, one obtains the commutative diagram
\DIAGV{50}
{} \n {} \n {} \n {} \n {I_0} \n {} \n {} \n {\earv{150}} \n {} \n {} \n {A_0} \n {} \n {} \n {\earv{150}} \n {} \n {} \n {D_0} \nn 
{} \n {} \n {} \n {\sweql} \n {} \n {} \n {} \n {} \n {} \n {\swaR{a^1}} \n {} \n {} \n {} \n {} \n {} \n {\Swar{d}} \n {} \nn 
{} \n {} \n {I_0} \n {} \n {} \n {\earv{150}} \n {} \n {} \n {A} \n {} \n {} \n {\earv{150}} \n {} \n {} \n {D_1} \nn 
{} \n {\swaR{i}} \n {} \n {} \n {\seqlv{170}} \n {} \n {} \n {\swaR{a^2}} \n {} \n {} \n {\saRv{f_0}{170}} \n {} \n {} \n {\sweql} \n {} \n {} \n {\saRv{g_0}{170}} \nn 
{I_1} \n {} \n {} \n {\earv{150}} \n {} \n {} \n {A_1} \n {} \n {} \n {\earv{150}} \n {} \n {} \n {D_1} \nn  
{} \n {} \n {\seqlv{170}} \n {} \n {} \n {} \n {} \n {} \n {\saRv{f'}{170}} \n {} \n {} \n {} \n {} \n {} \n {\saRv{g_1}{170}} \n {} \n {} \nn 
{} \n {} \n {} \n {} \n {I_0} \n {} \n {} \n {\earv{150}} \n {} \n {} \n {B_0} \n {} \n {} \n {\earv{150}} \n {} \n {} \n {J_0} \nn 
{\seqlv{170}} \n {} \n {} \n {\sweql} \n {} \n {} \n {\saRv{f_1}{170}} \n {} \n {} \n {\swaR{b^1}} \n {} \n {} \n {\saRv{g_1}{170}} \n {} \n {} \n {\swaR{j}} \nn 
{} \n {} \n {I_0} \n {} \n {} \n {\earv{150}} \n {} \n {} \n {B} \n {} \n {} \n {\earv{150}} \n {} \n {} \n {J_1} \nn 
{} \n {\swaR{i}} \n {} \n {} \n {} \n {} \n {} \n {\swaR{b^2}} \n {} \n {} \n {} \n {} \n {} \n {\sweql} \nn 
{I_1} \n {} \n {} \n {\earv{150}} \n {} \n {} \n {B_1} \n {} \n {} \n {\earv{150}} \n {} \n {} \n {J_1,}  
\diag
where the top level yields a mono-epi decomposition of the pullback of $\xi$ along $g.$
\end{proof}

We will use the notation $B = A \star C$ to indicate the existence of a conflation $A \to B \to C$ in an exact category $(\mcA; \mcE).$
If $i$ and $j$ are arrows in $\mcA,$ we say that an arrow $a$ in $\Arr (\mcA)$ is an extension of $j$ by $i,$ denoted $a = i \star j,$ if there exists an $\ME$-conflation $i \to a \to j.$ For example, if $A \to B \to C$ is a conflation in $\mcE,$ then the corresponding conflation 
$1_A \to 1_B \to 1_C$ in $\Arr (\mcE)$ is clearly an $\ME$-conflation, so that $1_B = 1_A \star 1_C$ holds in $(\Arr (\mcA); \ME).$ 

The following proposition is an application of Theorem~\ref{mono-epi}.

\begin{proposition} \label{associativity}
If $i,$ $j,$ and $k \in \Arr (\mcA),$ then $i \star (j \star k) = (i \star j) \star k.$
\end{proposition}

\begin{proof}
The statement of the proposition should be interpreted as saying that an arrow $a$ is of the form $i \star (j \star k)$ if and only if it is of the form $(i \star j) \star k.$ Consider the commutative diagram
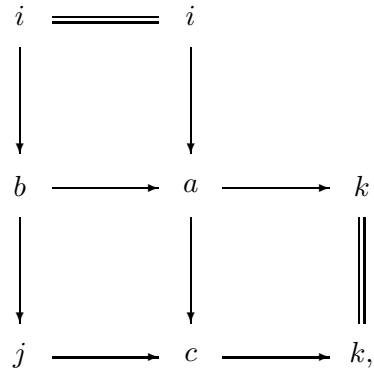
\DIAGV{80}
{i} \n {\eeql} \n {i} \nn
{\sar} \n {} \n {\sar} \nn
{b} \n {\ear} \n {a} \n {\ear} \n {k} \nn
{\sar} \n {} \n {\sar} \n {} \n {\seql} \nn
{j} \n {\ear} \n {c} \n {\ear} \n {k,} 
\diag
in $\Arr (\mcA).$ If $a = i \star (j \star k),$ then there is a diagram of this form, where the bottom row is an ME-conflation, so that $c = j \star k,$ and the middle column is an ME-conflation, so that $a = i \star c.$ By Axiom E$_1^{\op}$ for an exact category, the middle row is also an ME-conflation. The left column is also an ME-conflation, because it is obtained by pullback of the middle column along the inflation in the bottom row. The proof of the converse uses the dual argument.
\end{proof}

\section{Extension Ideals}

If $\mcM$ and $\mcN$ are classes of morphisms in $\mcA,$ then $\mcM \star \mcN$ denotes the class of morphisms that arise as extensions $a \star b,$ where $a \in \mcM$ and $b \in \mcN.$ Moreover, if $\mcK$ is a third class of morphisms, then the notation $\mcM \star \mcN \star \mcK$ is unambiguous, by Proposition~\ref{associativity}. If $\mcI$ and $\mcJ$ are ideals, then the ideal $\mcI \diamond \mcJ := \langle \mcI \star \mcJ \rangle$ is the {\em extension ideal} of $\mcJ$ by $\mcI.$ Because $i = i \star 0$ and $j = 0 \star j,$ the extension ideal $\mcI \diamond \mcJ$ contains both of the ideals $\mcI$ and $\mcJ.$ The elements of this extension ideal are described as follows. 

\begin{lemma} \label{ext crit}
Let $\mcI$ and $\mcJ$ be ideals of $\mcA.$ An arrow $a : A_0 \to A_1$ in $\mcA$ belongs to $\mcI \diamond \mcJ$ if and only if it satisfies one (resp., both) of the following equivalent conditions:
\begin{enumerate}
\item $a$ is a composition of morphisms $a : A_0 \stackrel{a^1}{\to} A \stackrel{a^2}{\to} A_1$ that are part of a commutative diagram
\DIAGV{90}
{} \n {} \n {} \n {A_0} \n {} \n {} \nn
{} \n {} \n {} \n {\saR{a^1}} \n {\Sear{j}} \n {} \nn
{\Xi:} \n {I} \n {\Ear{m}} \n {A} \n {\Ear{e}} \n {J} \nn
{} \n {} \n {\Sear{i}} \n {\saR{a^2}} \n {} \n {} \nn
{} \n {} \n {} \n {A_1,} \n {} \n {}
\diag
where $i \in \mcI,$ $j \in \mcJ,$ and the middle row is a conflation $\Xi$ in $\mcE;$
\item there are morphisms $r$ and $s$ in $\mcA$ such that $a = r(i \star j)s,$ where $i \in \mcI$ and $j \in \mcJ.$
\end{enumerate}
\end{lemma}

\begin{proof}
Let us prove that the morphisms $a$ that satisfy Condition ($1$) form an ideal that contains every extension $i \star j.$ If $a = i \star j$ is an extension of morphisms, then there is an ME-conflation $\xi : i \to a \to j.$ By Definition~\ref{mono epi def}, the arrow $a$ may be factored as 
$a = a^2 a^1$ with $je_0 \in \mcJ$ and $m_1i \in \mcI$ as required. It is easy to see that the morphisms satisying Condition ($1$) are closed under left and right multiplication. Finally, let us prove that if two parallel arrows $a_1,$ $a_2 : A_0 \to A_1$ possess a factorization satisfying Condition ($1$), then so does $a_1 + a_2 : A_0 \to A_1.$ We can factor $a_n,$ $n = 1,2$ as $a_n = a^2_n a^1_n,$ and there are commutative diagrams 
\DIAGV{90}
{} \n {} \n {} \n {A_0} \n {} \n {} \nn
{} \n {} \n {} \n {\saR{a^1_n}} \n {\Sear{j_n}} \n {} \nn
{\Xi_n:} \n {I^n} \n {\Ear{m_n}} \n {A^n} \n {\Ear{e_n}} \n {J^n} \nn
{} \n {} \n {\Sear{i_n}} \n {\saR{a^2_n}} \n {} \n {} \nn
{} \n {} \n {} \n {A_1,} \n {} \n {}
\diag
for $n = 1, 2.$ The $i_n$ belong to $\mcI,$ the $j_n$ to $\mcJ,$ and the $\Xi_n$ are conflations for $n = 1,2.$ By Proposition 2.9 of~\cite{B}, a direct sum of conflations is itself a conflation, so that
\DIAGV{150}
{} \n {} \n {} \n {A_0} \n {} \n {} \nn
{} \n {} \n {} \n {\Sar{\left( \begin{array}{c} a^1_1 \\ a^1_2 \end{array} \right)}} \n {\Sear{\left( \begin{array}{c} j_1 \\ j_2 \end{array} \right)}} \n {} \nn
{\Xi_1 \oplus \Xi_2:} \n {I^1 \oplus I^2} \n {\Ear{\left( \begin{array}{cc} m_1 & 0 \\ 0 & m_2 \end{array} \right)}} \n {A^1 \oplus A^2} \n {\Ear{\left( \begin{array}{cc} e_1 & 0 \\ 0 & e_2 \end{array} \right)}} \n {J^1 \oplus J^2} \nn
{} \n {} \n {\seaR{(i_1,i_2)}} \n {\saR{(a^2_1, a^2_2)}} \n {} \n {} \nn
{} \n {} \n {} \n {A_1} \n {} \n {}
\diag
yields a decomposition of $a_1 + a_2$ with $(i_1, i_2) \in \mcI$ and ${\dis \left( \begin{array}{c} j_1 \\ j_2 \end{array} \right) \in \mcJ}.$

\noindent ($1$) $\Rightarrow$ ($2$). Suppose now that the factorization $a : A_0 \stackrel{a^1}{\to} A \stackrel{a^2}{\to} A_1$ satisfies Condition ($1$). Take the pullback of $\Xi$ along $j$ and the pushout along $i$ to obtain 
\DIAG
{\Ext (j, I)(\Xi):} \n {I} \n {\ear} \n {A'_0} \n {\Ear{e_0}} \n {A_0} \nn
{} \n {\seql} \n {} \n {\saR{j'}} \n {\swaR{a^1}} \n {\saR{j}} \nn
{\Xi :} \n {I} \n {\Ear{m}} \n {A} \n {\Ear{e}} \n {J} \nn
{} \n {\saR{i}} \n {\swaR{a^2}} \n {\saR{i'}} \n {} \n {\seql} \nn
{\Ext (J, i)(\Xi):} \n {A_1} \n {\Ear{m_1}} \n {A'_1} \n {\ear} \n {J.}
\diag
Then $i'j' = i \star j.$ Because $j = e a^1$ and the top right commutative square is a pullback diagram, there is a section 
$s : A_0 \to A'_0$ of $e_0,$ $e_0s = 1_{A_0},$ such that $j's = a^1.$ Similarly, there is a retraction $r : A'_1 \to A_1$ of $m_1$ such that 
$ri' = a^2.$ Thus $a = a^2 a^1 = ri'j's =r(i \star j)s.$ \smallskip

\noindent Obviously, every morphism that satisfies Condition ($2$) belongs to $\mcI \diamond \mcJ.$
\end{proof}

In order to prove that the operation that associates to two ideals $\mcI$ and $\mcJ$ the extension ideal $\mcI \diamond \mcJ$ is associative, we will make use of the following observation.

\begin{lemma} \label{star composition}
If $i$ and $j$ are composable morphisms, and $k$ an arbitrary morphism, then for every extension $a = (ij) \star k,$ there is a morphism $j'$ such that  $a = (i \star k)j'.$
\end{lemma}

\begin{proof}
Consider a mono-epi factorization of an ME-conflation $\xi : ij \to a \to k$ 
\DIAGV{80}
{\Xi_0:} \n {J_0} \n {\ear} \n {A_0} \n {\ear} \n {K_0} \nn
{} \n {\seql} \n {} \n {\saR{a^1}} \n {} \n {\saR{k}} \nn
{\Xi:} \n {J_0} \n {\ear} \n {A} \n {\ear} \n {K_1} \nn
{} \n {\saR{j}} \n {} \n {\saR{a^2}} \n {} \n {\seql} \nn
{\Xi':} \n {J_1} \n {\ear} \n {A'} \n {\ear} \n {K_1} \nn
{} \n {\saR{i}} \n {} \n {\saR{a^3}} \n {} \n {\seql} \nn
{\Xi_1:} \n {I_1} \n {\ear} \n {A_1} \n {\ear} \n {K_1,}
\diag
where $a = a^3 a^2 a^1$ and the pushout of $\Xi$ along $ij$ has been factored as the composition of the pushout along $j$ followed by the pushout along $i.$ Now compose the top two morphisms of conflations and replace the composition with its pullback-pushout factorization to obtain
\DIAGV{80}
{\Xi_0:} \n {J_0} \n {\ear} \n {A_0} \n {\ear} \n {K_0} \nn
{} \n {\saR{j}} \n {} \n {\saR{j'}} \n {} \n {\seql} \nn
{} \n {J_1} \n {\ear} \n {A} \n {\ear} \n {K_0} \nn
{} \n {\seql} \n {} \n {\saR{k'}} \n {} \n {\saR{k}} \nn
{\Xi':} \n {J_1} \n {\ear} \n {A'} \n {\ear} \n {K_1} \nn
{} \n {\saR{i}} \n {} \n {\saR{a^3}} \n {} \n {\seql} \nn
{\Xi_1:} \n {I_1} \n {\ear} \n {A_1} \n {\ear} \n {K_1.}
\diag
Then $a = a^3 k' j' = (i \star k)j',$ as required.
\end{proof}

\begin{proposition} \label{ideal associativity}
If $\mcI,$ $\mcJ$ and $\mcK$ are ideals of $(\mcA; \mcE),$ then 
$(\mcI \diamond \mcJ) \diamond \mcK = \langle \mcI \star \mcJ \star \mcK \rangle = \mcI \diamond (\mcJ \diamond \mcK).$
\end{proposition} 

\begin{proof}
We only prove the first equality; the proof of the other is similar. By Lemma~\ref{ext crit}, every element of 
$\mcI \diamond \mcJ$ is of the form $a = r(i \star j)s,$ with $i \in \mcI$ and $j \in \mcJ.$ By Lemma~\ref{star composition}, if 
$k \in \mcK,$ then $a \star k = r(i \star j \star k)s',$ for some $s'.$ Thus 
$(\mcI \diamond \mcJ) \diamond \mcK \subseteq \langle \mcI \star \mcJ \star \mcK \rangle.$ 
The converse inclusion follows from Proposition~\ref{associativity}.
\end{proof}

If $\mcX$ and $\mcY$ are subcategories of $\mcA,$ then $\mcX \star \mcY$ denotes the subcategory of objects $Z$ that arise as the middle term of a conflation $\Xi : X \to Z \to Y$ in $(\mcA; \mcE).$

\begin{theorem} \label{object ext}
If $\mcI$ and $\mcJ$ are object ideals, then so is $\mcI \diamond \mcJ = \langle \Ob (\mcI) \star \Ob (\mcJ) \rangle.$  
In that case, $\Ob (\mcI \diamond \mcJ) = \add [\Ob (\mcI) \star \Ob (\mcJ)].$
\end{theorem}

\begin{proof}
Suppose that $I \in \Ob (\mcI)$ and $J \in \Ob (\mcJ),$ and consider an object $X = I \star J$ that is an extension of $J$ by $I.$ Then 
$1_X = 1_I \star 1_J$ belongs to $\mcI \star \mcJ.$ Thus $\Ob (\mcI) \star \Ob (\mcJ) \subseteq \Ob (\mcI \diamond \mcJ)$ and, in particular,
$\langle \Ob (\mcI) \star \Ob (\mcJ) \rangle \subseteq \mcI \diamond \mcJ.$

To prove the converse, consider an extension $a = i \star j$ with $i \in \mcI$ and $j \in \mcJ.$ By hypothesis, the morphism $i$ factors as
$i : I_0 \stackrel{i'}{\rightarrow} I \to I_1$ where $I$ is an object of $\mcI,$ and $j$ factors as 
$j : J_0 \to J \stackrel{j'}{\rightarrow} J_1$ through an object $J$ of $\mcJ.$ The mono-epi factorization of the ME-conflation $\xi : i \to a \to j$
factors further as
\DIAGV{80}
{\Xi_0} \n {I_0} \n {\ear} \n {A_0} \n {\ear} \n {J_0} \nn
{} \n {\seql} \n {} \n {\sar} \n {} \n {\sar} \nn
{\Xi_J :} \n {I_0} \n {\ear} \n {J'} \n {\ear} \n {J} \nn
{} \n {\seql} \n {} \n {\sar} \n {} \n {\saR{j'}} \nn
{\Xi:} \n {I_0} \n {\ear} \n {A} \n {\ear} \n {J_1} \nn
{} \n {\Sar{i'}} \n {} \n {\sar} \n {} \n {\seql} \nn
{\Xi_I :} \n {I} \n {\ear} \n {I'} \n {\ear} \n {J_1} \nn
{} \n {\sar} \n {} \n {\sar} \n {} \n {\seql} \nn
{\Xi_1} \n {I_1} \n {\ear} \n {A_1} \n {\ear} \n {J_1,}
\diag
where every row is a conflation. The extension $i' \star j'$ appears as the middle arrow of the morphism of conflations from $\Xi_J$ to $\Xi_I,$
whose pushout-pullback factorization is given by
\DIAGV{80}
{\Xi_J :} \n {I_0} \n {\ear} \n {J'} \n {\ear} \n {J} \nn
{} \n {\Sar{i'}} \n {} \n {\sar} \n {} \n {\seql} \nn
{} \n {I} \n {\ear} \n {A'} \n {\ear} \n {J} \nn
{} \n {\seql} \n {} \n {\sar} \n {} \n {\saR{j'}} \nn
{\Xi_I :} \n {I} \n {\ear} \n {I'} \n {\ear} \n {J_1.}
\diag
This proves that $i' \star j',$ and therefore $i \star j,$ factors through $A' = I \star J,$ which belongs to \linebreak 
$\Ob (\mcI) \star \Ob (\mcJ).$ Thus $\mcI \star \mcJ \subseteq \langle \Ob (\mcI) \star \Ob (\mcJ) \rangle,$ and the equality is proved.
The last statement is immediate from the equality. It is intended to emphasize that, while the subcategory $\Ob (\mcI) \star \Ob (\mcJ)$
is closed under finite direct sums, it need not be closed under direct summands.
\end{proof}

\section{$\Ext$-Orthogonality}

Let $(B,A)$ be a pair of objects in an exact category $(\mcA; \mcE).$ Two conflations of the form \linebreak 
$\Xi_i : A \to C_i \to B,$ $i = 0,$ $1,$ are said to be equivalent if there exists an isomorphism $\xi : \Xi_0 \to \Xi_1$ of the form 
\DIAGV{80}
{\Xi_0 :} \n {A} \n {\ear} \n {C_0} \n {\ear} \n {B} \nn
{} \n {\saR{1_A}} \n {} \n {\sar} \n {} \n {\saR{1_B}} \nn
{\Xi_1 :} \n {A} \n {\ear} \n {C_1} \n {\ear} \n {B.}
\diag  
The equivalence classes form a class $\Ext (B,A) := \Ext_{\mcA}(B,A)$ that acquires the structure of an abelian group, with respect to the Baer sum operation. If $j : B' \to B$ is a morphism, then the pullback of $\Xi \in \Ext (B,A)$ yields an element $\Ext (j,A)(\Xi) \in \Ext (B', A).$ Similarly, 
if $i : A \to A',$ then the pushout yields the conflation $\Ext (B,i)(\Xi) \in \Ext (B,A').$ These properties define a bifunctor
$$\Ext : \mcA^{\op} \times \mcA \to \Ab.$$ 

\begin{definition}
A pair $(j,i)$ of morphisms in $\mcA$ is $\Ext${\em -orthogonal,} denoted $\Ext(j,i) = 0,$ if every ME-extension $\xi : i \to a \to j$ in $\Arr (\mcA)$ is {\em null-homotopic.} This means that there are morphisms $h : A_0 \to I_1$ and $g : J_0 \to A_1$ as in the diagram
\DIAGV{80}
{\Xi_0 :} \n {I_0} \n {\ear} \n {A_0} \n {\Ear{e_0}} \n {J_0} \nn
{} \n {\saR{i}} \n {\swaR{h}} \n {\saR{a}} \n {\swaR{g}} \n {\saR{j}} \nn
{\Xi_1 :} \n {I_1} \n {\Ear{m_1}} \n {A_1} \n {\ear} \n {J_1}
\diag 
satisfying $a = m_1h + ge_0.$
\end{definition}

\noindent {\bf Caution: $\Ext$-orthogonality for a pair of morphisms $(i,j)$ is properly weaker than the condition $\Ext (j,i) = 0$ in the exact category $(\Arr (\mcA); \Arr (\mcE)),$ and therefore in $(\Arr (\mcA); \ME),$ which we will not use in this paper (see~\cite[Remark 3.4]{EGO}).} Indeed, the next proposition shows that the definition of $\Ext$-orthogonality given above is equivalent to the definition of $\Ext$-orthogonality introduced in~\cite{FGHT}.

\begin{proposition}
If $i:B' \to B$ and $j:A \to A'$ are morphisms in $\mcA,$ then the pair $(j,i)$ of morphisms in $(\mcA; \mcE)$ is $\Ext$-orthogonal if and only if the induced morphism $$\Ext (j,i) : \Ext (B,A) \to \Ext (B',A')$$ of abelian groups is zero.
\end{proposition}

\begin{proof}
Every ME-conflation $\xi : i \to c \to j$ has a mono-epi factorization of the form
\DIAG
{\Ext (j,A) (\Xi):} \n {A} \n {\ear} \n {C_0} \n {\ear} \n {B'} \nn
{} \n {\seql} \n {} \n {\sar} \n {} \n {\saR{j}} \nn
{\Xi :} \n {A} \n {\ear} \n {C} \n {\ear} \n {B} \nn
{} \n {\saR{i}} \n {} \n {\sar} \n {} \n {\seql} \nn
{\Ext (B,i) (\Xi):} \n {A'} \n {\ear} \n {C_1} \n {\ear} \n {B} 
\diag
for some $\Xi \in \Ext (B,A),$ and every $\Xi \in \Ext (B,A)$ gives rise in this manner to an ME-conflation $\xi : i \to c \to j.$ The pullback-pushout factorization factors through the conflation $\Ext (i,j)(\Xi)$ (see~\cite[Prop 3]{FGHT} for a more thorough explanation). But $\Ext (i,j)(\Xi)$ is split if and only if $\xi$ is null-homotopic. 
\end{proof}

If $\mcI$ is an ideal, then the ideal {\em right} $\Ext${\em -perpendicular} to $\mcI$ is defined to be
$$\mcI^{\perp} = \{ j \; | \; \Ext (i,j) = 0 \; \mbox{for all} \; i \in I \}.$$
If $\mcJ$ is an ideal, the {\em left} $\Ext${\em -perpendicular} ideal ${^{\perp}}\mcJ$ is defined dually.

\begin{theorem} \label{mult/ext}
If $\mcI$ and $\mcJ$ are ideals, then $(\mcI \mcJ)^{\perp} \supseteq \mcJ^{\perp} \diamond \mcI^{\perp}.$ 
\end{theorem}

\begin{proof}
By Proposition~\ref{ext crit}, a morphism $c : C_0 \to C_1$ in $\mcJ^{\perp} \diamond \mcI^{\perp}$ may be expressed as a composition 
$c = c^2 c^1$ given by the commutative diagram
\DIAGV{80}
{} \n {} \n {} \n {C_0} \n {} \n {} \nn
{} \n {} \n {} \n {\saR{c^1}} \n {\Sear{i^{\perp}}} \n {} \nn
{\Xi:} \n {J'} \n {\Ear{m}} \n {C} \n {\Ear{e}} \n {I'} \nn
{} \n {} \n {\seaR{j^{\perp}}} \n {\saR{c^2}} \n {} \n {} \nn
{} \n {} \n {} \n {C_1} \n {} \n {}
\diag
where $\Xi$ is a conflation, $i^{\perp} \in \mcI^{\perp}$ and $j^{\perp} \in \mcJ^{\perp}.$ Let $i: I_0 \to I_1$ be a morphism in $\mcI$ and 
$j : J_0 \to I_0$ a morphism in $\mcJ$ and apply the transformation $\Ext (ij, -) = \Ext (j, -) \Ext (i, -)$ to obtain the commutative diagram
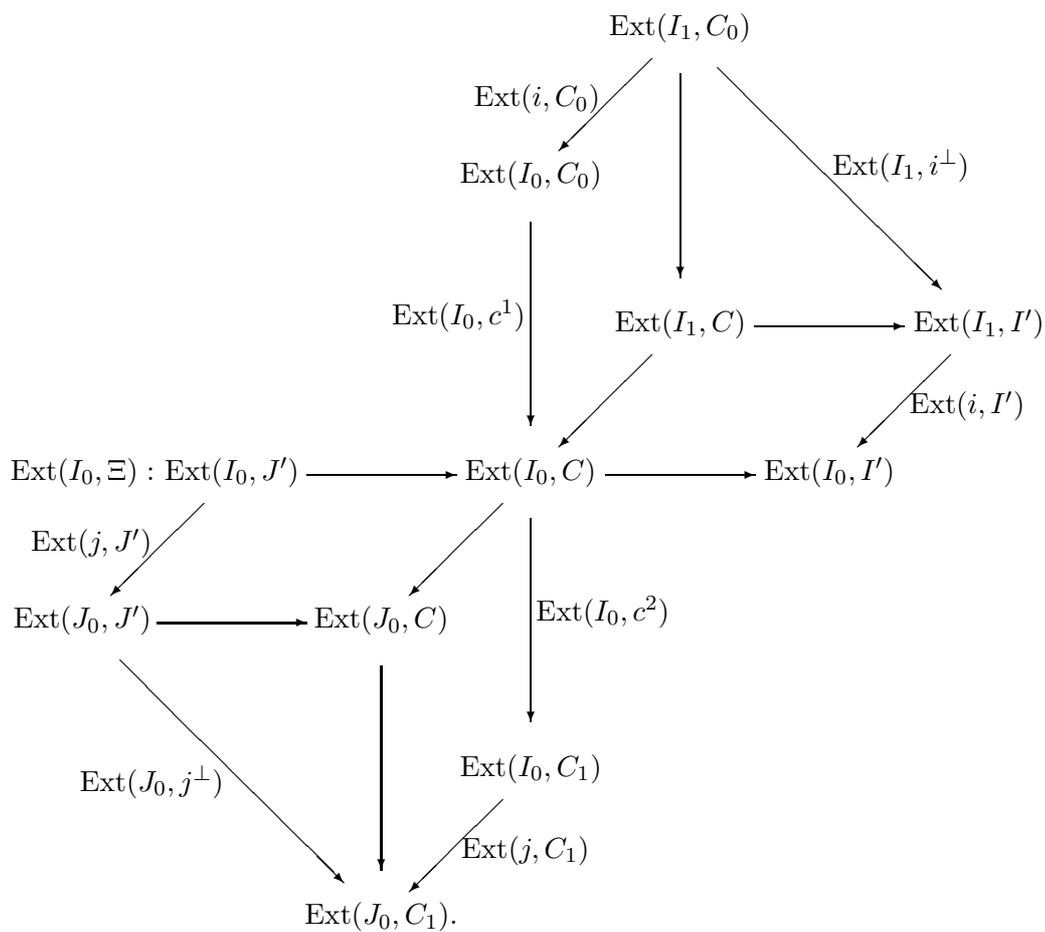
\DIAGV{70}
{} \n {} \n {} \n {} \n {} \n {} \n {} \n {} \n {\Ext (I_1, C_0)} \n {} \n {} \n {} \n {} \nn
{} \n {} \n {} \n {} \n {} \n {} \n {} \n {\Swarv{\Ext (i, C_0)}{50}} \n {} \n {} \n {} \n {} \n {} \nn
{} \n {} \n {} \n {} \n {} \n {} \n {\Ext (I_0, C_0)} \n {} \n {\sarv{110}} \n {} \n {\Searv{\Ext (I_1, i^{\perp})}{120}} \n {} \n {} \nn
{} \n {} \n {} \n {} \n {} \n {} \n {} \n {} \n {} \n {} \n {} \n {} \n {} \nn
{} \n {} \n {} \n {} \n {} \n {} \n {\Sarv{\Ext (I_0, c^1)}{110}} \n {} \n {\Ext (I_1, C)} \n {} \n {\earv{80}} \n {} \n {\Ext (I_1, I')} \nn
{} \n {} \n {} \n {} \n {} \n {} \n {} \n {\swarv{50}} \n {} \n {} \n {} \n {\swaRv{\Ext (i, I')}{50}} \n {} \nn
{\Ext (I_0, \Xi) :} \n {} \n {\Ext (I_0, J')} \n {} \n {\earv{80}} \n {} \n {\Ext (I_0, C)} \n {} \n {\earv{80}} \n {} \n {\Ext (I_0, I')} \nn
{} \n {\Swarv{\Ext (j, J')}{50}} \n {} \n {} \n {} \n {\swarv{50}} \nn
{\Ext (J_0, J')} \n {} \n {\earv{80}} \n {} \n {\Ext (J_0, C)} \n {} \n {\saRv{\Ext (I_0, c^2)}{110}} \nn
{} \nn
{} \n {} \n {\seaRv{\Ext (J_0, j^{\perp})}{120}} \n {} \n {\sarv{110}} \n {} \n {\Ext (I_0, C_1)} \nn
{} \n {} \n {} \n {} \n {} \n {\swaRv{\Ext (j, C_1)}{50}} \nn
{} \n {} \n {} \n {} \n {\Ext (J_0, C_1).} 
\diag 
Compose the labeled arrows to obtain the commutative diagram
\DIAGV{70}
{} \n {} \n {} \n {} \n {} \n {} \n {\Ext (I_1, C_0)}  \nn
{} \nn
{} \n {} \n {} \n {} \n {} \n {} \n {\Sarv{\Ext (i, c^1)}{110}} \n {} \n {\Searv{\Ext (i,i^{\perp})}{120}} \nn
{} \nn
{\Ext (I_0, \Xi) :} \n {} \n {\Ext (I_0, J')} \n {} \n {\Earv{\Ext (I_0,m)}{80}} \n {} \n {\Ext (I_0, C)} \n {} \n {\Earv{\Ext (I_0, e)}{80}} \n {} \n {\Ext (I_0, I')} \nn
{} \nn
{} \n {} \n {} \n {} \n {\seaRv{\Ext (j,j^{\perp})}{120}} \n {} \n {\saRv{\Ext (j, c^2)}{110}} \nn
{} \nn
{} \n {} \n {} \n {} \n {} \n {} \n {\Ext (J_0, C_1)} 
\diag
Because $\Ext (i, i^{\perp}) = 0,$ we get that $\im \Ext (i, c^1) \subseteq \Ker \Ext (I_0, e).$ Similarly, the hypothesis that $\Ext (j,j^{\perp}) = 0$ implies $\im \Ext (I_0, m) \subseteq \Ker \Ext (j, c^2).$ The middle row is exact, so it follows that 
$\Ext (ij, c) = \Ext (ij, c^2 c^1) = \Ext (j, c^2) \Ext (i, c^1) =  0.$
\end{proof}

The ideal $\Hom$ consists of all morphisms in $\mcA.$ A morphism $i : E \to E'$ is {\em injective} if it belongs to the right perpendicular ideal $\Hom^{\perp},$ denoted by $\mcEinj.$ Thus $\Ext (-,i) = 0,$ which means that for every conflation $\Xi : E \to C \to B,$ the pushout 
$\Ext (B, i)(\Xi )$ of $\Xi$ along $i$ is split. As a consequence of Theorem~\ref{mult/ext}, a right $\Ext$-perpendicular ideal satisfies the following closure property (cf.~\cite[Prop 9]{FGHT}).

\begin{corollary} \label{inj ext}
If $\mcI$ is an ideal in $\mcA,$ then $\mcI^{\perp} = (\Hom \; \mcI)^{\perp} = \mcI^{\perp} \diamond \mcEinj.$
\end{corollary}

An ideal $\mcI$ is {\em idempotent} if $\mcI^2 = \mcI.$ In that case, Theorem~\ref{mult/ext} implies that  
$\mcI^{\perp} = (\mcI^2 )^{\perp} \supseteq \mcI^{\perp} \diamond \mcI^{\perp}.$ An ideal $\mcJ$ is {\em closed under extensions}
if $\mcJ \diamond \mcJ = \mcJ.$

\begin{corollary} \label{idem 1}
If $\mcI$ is an idempotent ideal, then $\mcI^{\perp}$ is closed under extensions.
\end{corollary}

\section{Salce's Lemma}

Recall from the Introduction that a {\em special} $\mcI$-precover of an object $A \in \mcA$ is a morphism \linebreak $i_1 : I_1 \to A$ in $\mcI$ that arises from
a pushout 
\DIAGV{70}
{\Xi_0 :} \n {K_0} \n {\ear} \n {C_0} \n {\ear} \n {A} \nn
{} \n {\saR{k}} \n {} \n {\saR{c}} \n {} \n {\seql} \nn
{\Xi_1 :} \n {K_1} \n {\ear} \n {C_1} \n {\Ear{i_1}} \n {A} 
\diag
along a morphism $k \in \mcI^{\perp}.$ The morphism $k$ is then called the $\mcI$-syzygy of $A$ and is denoted by $k = \gro_{\mcI} (A)$ or, for brevity, just $\gro (A).$ A special $\mcI$-precover of $A$ is therefore a morphism $i_1 : C_1 \to A$ in $\mcI$ that is part of an 
$\Arr (\mcE)$-conflation of the form
\DIAGV{70}
{\xi:} \n {\gro (A)} \n {\ear} \n {c} \n {\Ear{i}} \n {1_A,}
\diag
where $\gro(A) \in \mcI^{\perp}.$ Because the right term is $1_A,$ the conflation is an ME-conflation.

\begin{definition} \label{spi}
An ideal $\mcI$ of $\mcA$ is a {\em special precovering} ideal if every object in $\mcA$ has a special $\mcI$-precover. An ideal 
$\mcJ \subseteq \mcI^{\perp}$ is an $\mcI${\em -syzygy ideal} if it contains an $\mcI$-syzygy $\gro (A),$ for every object $A \in \mcA.$ Such an ideal will be denoted by $\gro (\mcI).$
\end{definition}

For example if an ideal $\mcI$ is special precovering, then $\mcI^{\perp} = \gro (\mcI)$ is the largest $\mcI$-syzygy ideal. The proof of the following proposition implies~\cite[Prop 11]{FGHT} that a special $\mcI$-precover of an object $A$ is an $\mcI$-precover.

\begin{proposition} \label{double dual}
If $\mcI$ is a special precovering ideal of $(\mcA; \mcE),$ and $\gro (\mcI)$ an $\mcI$-syzygy ideal, then ${^{\perp}}\gro (\mcI) = \mcI.$
\end{proposition}

\begin{proof}
Because $\gro (\mcI) \subseteq \mcI^{\perp},$ it follows certainly that $\mcI \subseteq {^{\perp}}\gro (\mcI).$ To prove the converse inclusion,
let $A \in \mcA$ and consider a special $\mcI$-precover $i_1 : C_1 \to A$ as above, and take the pullback of $\Xi_0$ along 
$i' \in {^{\perp}}\gro (\mcI),$
\DIAGV{70}
{\Xi'_0 :} \n {K_0} \n {\ear} \n {C'} \n {\ear} \n {I'} \nn
{} \n {\seql} \n {} \n {\saR{c'}} \n {} \n {\saR{i'}} \nn
{\Xi_0 :} \n {K_0} \n {\ear} \n {C_0} \n {\ear} \n {A} \nn
{} \n {\saR{\gro (A)}} \n {} \n {\saR{c}} \n {} \n {\seql} \nn
{\Xi_1 :} \n {K_1} \n {\ear} \n {C_1} \n {\Ear{i_1}} \n {A.} 
\diag
This is an ME-conflation of the form $\gro (A) \to cc' \to i'.$ As $\Ext_{\mcA} (i', \gro (A)) = 0,$ this conflation is null-homotopic. The homotopy  then yields a factorization
\DIAGV{80}
{} \n {} \n {} \n {} \n {I'} \nn
{} \n {} \n {} \n {\Swar{g}} \n {\saR{i'}} \nn
{K_1} \n {\ear} \n {C_1} \n {\Ear{i_1}} \n {A,}
\diag
which implies $i' = i_1g \in \mcI.$ 
\end{proof}

Given an ideal $\mcJ,$ the notion of a {\em special} $\mcJ${\em -preenvelope} is defined dually. The ideal $\mcJ$ is a {\em special preenveloping} ideal if every object $B$ in $\mcA$ has a special $\mcJ$-preenvelope. A pair of ideals $(\mcI, \mcJ)$ is an {\em ideal cotorsion pair} if $\mcJ = \mcI^{\perp}$ and $\mcI = {^{\perp}}\mcJ.$ Proposition~\ref{double dual} implies that if $\mcI$ is a special precovering ideal, then the ideal pair $(\mcI, \mcI^{\perp})$ is an ideal cotorsion pair that is {\em cogenerated} by $\gro (\mcI),$ in the sense that 
$(\mcI, \mcI^{\perp}) = ({^{\perp}}\gro (\mcI), ({^{\perp}}\gro (\mcI))^{\perp}).$ An ideal cotorsion pair $(\mcI, \mcJ)$ is {\em complete} if $\mcI$ is special precovering and $\mcJ$ is special preenveloping. The next result is Salce's Lemma, which implies that if $\mcI$ is a special precovering ideal, then the ideal cotorsion pair $(\mcI, \mcI^{\perp})$ is complete. It generalizes the implication ($2$) $\Rightarrow$ ($3$) 
of~\cite[Theorem 1]{FGHT}, by weakening the hypothesis to one that is self dual.

Recall that the exact category $(\mcA; \mcE)$ {\em has enough injective morphisms} if for every object $A \in \mcA,$ there is an injective inflation 
$e : A \to E.$ The notion of a {\em projective morphism} and that of an exact category {\em having enough projective morphisms} are defined dually. 
 
\begin{theorem} \label{Salce} {\rm (Salce's Lemma)}
Let $(\mcA; \mcE)$ be an exact category with enough injective morphisms and enough projective morphisms. The rule $\mcI \mapsto \mcI^{\perp}$ is a bijective correspondence between the class of special precovering ideals $\mcI$ of $(\mcA; \mcE)$ and that of its special preenveloping ideals $\mcJ.$ The inverse rule is given by $\mcJ \mapsto {^{\perp}}\mcJ.$  
\end{theorem}

\begin{proof}
We use the hypothesis that there exist enough injective morphisms to prove that if $\mcI$ is a special precovering ideal, then $\mcI^{\perp}$ is a special preenveloping ideal. The proof that if $\mcJ$ is a special preenveloping ideal, then ${^{\perp}}\mcJ$ is a special precovering ideal is dual; it uses the dual hypothesis that there are enough projective morphisms. That the inverse rule is given by $\mcJ \mapsto {^{\perp}}\mcJ$ follows from Proposition~\ref{double dual}, because $\gro (\mcI) \subseteq \mcI^{\perp}.$

Let us proceed as in the proof of~\cite[Thm 18]{FGHT}. Given an object $A \in \mcA,$ we construct a special $\mcI^{\perp}$-preenvelope of $A.$ There is a conflation $\Xi : A \stackrel{e}{\rightarrow} E \to N,$ where $e : A \to E$ is an injective morphism. The cokernel $N$ has a special $\mcI$-precover $i_1 : C_1 \to N$ that arises as part of an ME-conflation $\gro (N) \to c \stackrel{i}{\rightarrow} 1_N.$ Take the pullback in 
$(\Arr (\mcA); \Arr (\mcE))$ of $1_{\Xi} : 1_A \stackrel{e}{\rightarrow} 1_E \to 1_N$ along $i : c \to 1_N$ to obtain
\DIAGV{80}
{} \n {} \n {} \n {\gro (N)} \n {\eeqlv{40}} \n {\gro (N)} \nn
{} \n {} \n {} \n {\saR{k}} \n {} \n {\sar} \nn
{} \n {1_A} \n {\Ear{j}} \n {b} \n {\ear} \n {c} \nn
{} \n {\seql} \n {} \n {\sar} \n {} \n {\saR{i}} \nn
{1_{\Xi}:} \n {1_A} \n {\Ear{e}} \n {1_E} \n {\ear} \n {1_N.}
\diag
This construction illustrates Theorem~\ref{mono-epi} nicely, as all the rows and columns are evidently ME-conflations. Let us regard this commutative diagram as a diagram in $\mcA,$
\DIAGV{60}
{} \n {} \n {} \n {} \n {} \n {} \n {W_0} \n {} \n {\eeqlv{100}} \n {} \n {W_0} \nn
{} \n {} \n {} \n {} \n {} \n {\Swar{\gro (N)}} \n {} \n {} \n {} \n {\Swar{\gro (N)}} \n {} \nn
{} \n {} \n {} \n {} \n {W_1} \n {} \n {\cross{\saRv{\movename{k_0}{0}{10}}{100}}{\eeqlv{100}}} \n {} \n {W_1} \n {} \n {\sarv{100}} \nn
{} \n {} \n {} \n {} \n {} \n {} \n {} \n {} \n {} \n {} \n {} \nn
{} \n {} \n {A} \n {} \n {\cross{\Earv{\movename{j_0}{20}{0}}{100}}{\saRv{\movename{k_1}{0}{10}}{100}}} \n {} \n {B_0} \n {} \n {\cross{\sarv{100}}{\earv{100}}} \n {} \n {C_0} \nn
{} \n {\sweql} \n {} \n {} \n {} \n {\swaR{b}} \n {} \n {} \n {} \n {\swaR{c}} \n {} \nn
{A} \n {} \n {{\cross{\Earv{\movename{j_1}{15}{0}}{100}}{\seqlv{100}}}} \n {} \n {B_1} \n {} \n {\cross{\earv{100}}{\sarv{100}}} \n {} \n {C_1} \n {} \n {\saRv{\movename{i_0}{0}{-10}}{100}} \nn
{} \nn
{\seqlv{100}} \n {} \n {A} \n {} \n {\cross{\sarv{100}}{\earv{100}}} \n {} \n {E} \n {} \n {\cross{\saRv{\movename{i_1}{0}{-10}}{100}}{\earv{100}}} \n {} \n {N} \nn
{} \n {\sweql} \n {} \n {} \n {} \n {\sweql} \n {} \n {} \n {} \n {\sweql} \n {} \nn
{A} \n {} \n {\Earv{e}{100}} \n {} \n {E} \n {} \n {\earv{100}} \n {} \n {N.} {} \n {}
\diag
We claim that the morphism $j_1 : A \to B_1$ is an $\mcI^{\perp}$-special preenvelope. Because it is obtained by pullback along $i_1 \in \mcI,$ it is enough to verify that $j_1 \in \mcI^{\perp}.$ Let us extract from the diagram above the commutative diagram
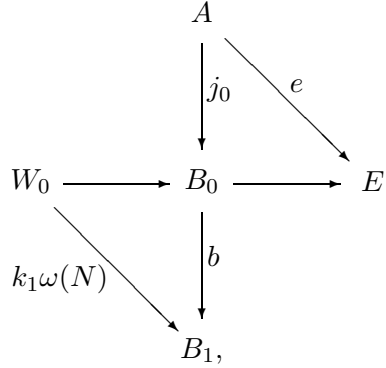
\DIAGV{80}
{} \n {} \n {A} \n {} \n {} \nn
{} \n {} \n {\saR{j_0}} \n {\Sear{e}} \n {} \nn
{W_0} \n {\ear} \n {B_0} \n {\ear} \n {E} \nn
{} \n {\seaR{k_1 \gro(N)}} \n {\saR{b}} \n {} \n {} \nn
{} \n {} \n {B_1,}
\diag
where the middle row is the conflation that appears in the back middle column of the previous diagram. Because $k_1 \gro (N) \in \mcI^{\perp}$ and 
$e \in \mcEinj,$ Proposition~\ref{ext crit} implies that $j_1$ belongs to $\mcI^{\perp} \diamond  \mcEinj.$ By Corollary~\ref{inj ext}, this latter ideal is contained in $\mcI^{\perp},$ as required.
\end{proof}

In the proof of Salce's Lemma, the morphism $\gro (N)$ may be taken from a given $\mcI$-syzygy ideal $\gro (\mcI).$ The $\mcI^{\perp}$-preenvelope 
$j_1 : A \to B_1$ constructed in that proof then belongs to $\gro (\mcI) \diamond \mcEinj.$ This implies that every morphism in $\mcI^{\perp}$ whose domain is $A$ factors through $j_1$ and, therefore, belongs to $\gro (\mcI) \diamond \mcEinj.$ Thus 
$\mcI^{\perp} = \gro (\mcI) \diamond \mcEinj,$ for every $\mcI$-syzygy ideal $\gro (\mcI).$ Corollary~\ref{inj ext}, on the other hand, implies that if $\mcJ \subseteq \mcI^{\perp}$ is an ideal, then $\mcJ \diamond \mcEinj$ is also contained in $\mcI^{\perp}.$ In view of that corollary, the equation $\mcI^{\perp} = \gro (\mcI) \diamond \mcEinj$ expresses that every $\mcI$-syzygy ideal $\gro (\mcI)$ nearly generates the ideal $\mcI^{\perp}.$ It turns out that this property characterizes $\mcI$-syzygy ideals.

\begin{theorem} \label{syzygy ideal}
Let $\mcI$ be a special precovering ideal of an exact category $(\mcA; \mcE)$ with enough injective morphisms. An ideal $\mcJ \subseteq \mcI^{\perp}$ is an $\mcI$-syzygy ideal if and only if $\mcJ \diamond \mcEinj = \mcI^{\perp}.$ 
\end{theorem}

\begin{proof}
One direction of the equivalence has just been established, so suppose that the ideal $\mcJ$ satisfies the equality 
$\mcJ \diamond \mcEinj = \mcI^{\perp}$ and let $A \in \mcA.$ There is a special $\mcI$-precover $i_1 : C_1 \to A$
\DIAGV{70}
{\Xi_0:} \n {W_0} \n {\ear} \n {C_0} \n {\ear} \n {A} \nn
{} \n {\saR{\gro}} \n {} \n {\saR{c}} \n {} \n {\seql} \nn
{\Xi_1:} \n {W_1} \n {\ear} \n {C_1} \n {\Ear{i_1}} \n {A,} 
\diag
where $\gro \in \mcI^{\perp} = \mcJ \diamond \mcEinj$ is a given $\mcI$-syzygy of $A.$ By Lemma~\ref{ext crit}, the morphism 
$\gro : W_0 \to W_1$ may be expressed as a composition, shown in the middle column of 
\DIAGV{90}
{} \n {} \n {W_0} \n {} \n {} \nn
{} \n {} \n {\saR{\gro^1}} \n {\Sear{e}} \n {} \nn
{J} \n {\Ear{m}} \n {W} \n {\Ear{p}} \n {E} \nn
{} \n {\seaR{j}} \n {\saR{\gro^2}} \n {} \n {} \nn
{} \n {} \n {W_1,}
\diag
where $j : J \to W_1$ belongs to $\mcJ,$ $e$ is an injective morphism and the middle row is a conflation in $(\mcA; \mcE).$ It suffices to verify that $j$ is itself an $\mcI$-syzygy of $A.$ Let us show, moreover, that $\Xi_1 \in \Ext (A,W_1)$ arises as the pushout along $j$ of some conflation in 
$\Ext (A, J).$ Apply the covariant functor $\Ext (A,-)$ to the preceding diagram to obtain
\DIAGV{70}
{} \n {} \n {} \n {} \n {\Ext (A,W_0)} \nn
{} \nn
{} \n {} \n {} \n {} \n {\Sarv{\Ext (A,\gro^1)}{110}} \n {} \n {\Searv{0}{120}} \nn
{} \nn
{\Ext (A,J)} \n {} \n {\Earv{\Ext (A,m)}{80}} \n {} \n {\Ext (A,W)} \n {} \n {\Earv{\Ext (A,p)}{80}} \n {} \n {\Ext (A,E)} \nn
{} \nn
{} \n {} \n {\seaRv{\Ext (A,j)}{120}} \n {} \n {\saRv{\Ext (A,\gro^2)}{110}} \nn
{} \nn
{} \n {} \n {} \n {} \n {\Ext (A,W_1)}
\diag
and note that $\Ext (A,e) = 0.$ The middle row is exact, so that $\Ext (A,\gro^1)(\Xi_0)$ belongs to the image of $\Ext (A,m).$ If 
$\Upsilon \in \Ext (A,J)$ is a preimage, then $$\Ext (A,j)(\Upsilon ) = \Ext (A,\gro^2)\Ext (A,m)(\Upsilon ) = 
\Ext (A,\gro^2 ) \Ext (A,\gro^1) (\Xi_0) = \Ext (A,\gro ) (\Xi_0) = \Xi_1,$$ 
as claimed. Thus $j = \gro_{\mcI} (A)$ is an $\mcI$-syzygy of $A.$
\end{proof}

\section{Object-Special Precovers}

Let $A$ be an object of $(\mcA; \mcE)$ and $\mcI$ an ideal. A special $\mcI$-precover of $A$ is said to be an {\em object-special} $\mcI${\em -precover} of $A$ if there is an $\mcI$-syzygy $\gro (A)$ of $A$ that is an isomorphism. Then there is an object, let us denote it by 
$\grO (A),$ such that $\gro (A) \isom 1_{\grO (A)}.$ A special $\mcI$-precover $i'_1 : C'_1 \to A$ appears as part of the ME-conflation in the top row of
\DIAGV{70}
{\xi':} \n {\gro (A)} \n {\ear} \n {c'} \n {\Ear{i'}} \n {1_A} \nn
{} \n {\saR{f}} \n {} \n {\sar} \n {} \n {\seql} \nn
{\xi:} \n {1_{\grO (A)}} \n {\ear} \n {c} \n {\Ear{i}} \n {1_A.}
\diag
Taking the pushout of $\xi'$ in $(\Arr (\mcA); \ME)$ along an isomorphism $f : \gro (A) \to 1_{\grO (A)}$ yields the ME-conflation $\xi,$ which is given by 
\DIAGV{70}
{\Xi_0:} \n {\grO (A)} \n {\ear} \n {C_0} \n {\ear} \n {A} \nn
{} \n {\seql} \n {} \n {\saR{c}} \n {} \n {\seql} \nn
{\Xi_1:} \n {\grO (A)} \n {\ear} \n {C_1} \n {\Ear{i_1}} \n {A.}
\diag
The kernel $\grO (A)$ of $i_1 : C_1 \to A$ belongs to $\Ob (\mcI^{\perp})$ and is called an {\em object} $\mcI$-{\em syzygy} of $A.$ To avoid confusion, the object $\mcI$-syzygy of $A$ may be denoted more precisely as $\grO_{\mcI}(A).$

\begin{definition} \label{ospi}
An ideal $\mcI$ is an {\em object-special precovering} ideal if every object $A$ in $\mcA$ has an object-special $\mcI$-precover.
\end{definition}

An object $E$ of $\mcA$ is {\em injective} if the morphism $1_E : E \to E$ is injective. The subcategory of $\mcA$ of injective objects is denoted by $\mcEInj := \Ob (\mcEinj).$ We say that the category $(\mcA; \mcE)$ has enough injective objects if for every object $A,$ there exists an inflation 
$e : A \to E$ with $E$ injective.

\begin{proposition} \label{object syzygy}
A special precovering ideal $\mcI$ is object-special precovering if and only if some $\mcI$-syzygy ideal $\gro (\mcI)$ is an object ideal.
If the category $(\mcA; \mcE)$ has enough injective objects, this is equivalent to the ideal $\mcI^{\perp}$ being an object ideal.
\end{proposition}

\begin{proof}
If $\mcI$ is an object-special precovering ideal, take $\gro (\mcI)$ to be any object ideal $\langle \grO (A) \; | \; A \in \mcA \rangle$ generated by object $\mcI$-syzygies. Conversely, if some $\mcI$-syzygy ideal $\gro (\mcI)$ is an object ideal, then it is possible to find, for every $A \in \mcA$ an $\mcI$-syzygy that factors through an object $\grO (A)$ in $\Ob (\mcI^{\perp}).$ The proof of Proposition 25 of~\cite{FGHT} shows then how to construct a deflation $i : C \to A$ in $\mcI$ with kernel $\grO (A).$

If $(\mcA; \mcE)$ has enough injective objects, and $\gro (\mcI)$ is an object ideal, then Theorem~\ref{object ext} implies that 
$\gro (\mcI) \diamond \mcEinj$ is itself an object ideal. By Theorem~\ref{syzygy ideal}, $\gro (\mcI) \diamond \mcEinj = \mcI^{\perp}.$
\end{proof}

A subcategory $\mcC$ of $\mcA$ that is closed under finite direct sums is an $\mcI${\em -syzygy subcategory} if it generates an $\mcI$-syzygy ideal, $\langle \mcC \rangle = \gro (\mcI).$ An $\mcI$-syzygy subcategory will be denoted by $\grO (\mcI).$

\begin{proposition} \label{ext by objects}
Suppose that $(\mcA; \mcE)$ has enough injective objects and that $\mcI$ is an object-special precovering ideal in $\mcA.$ A subcategory $\mcC$ of 
$\Ob (\mcI^{\perp})$ that is closed under finite direct sums is an $\mcI$-syzygy subcategory if and only if 
$\add (\mcC \star \mcEInj) = \Ob (\mcI^{\perp}).$
\end{proposition}

\begin{proof}
If $\grO (\mcI)$ is an $\mcI$-syzygy subcategory, then $\add (\grO (\mcI) \star \mcEInj) = \Ob (\mcI^{\perp}),$ by~\cite[Thm 27]{FGHT}. Conversely, suppose that $\mcC$ is a subcategory of $\Ob (\mcI^{\perp}),$ closed under finite direct sums, and satisfying 
$\add (\mcC \star \mcEInj) = \Ob (\mcI^{\perp}).$ Then
\begin{eqnarray*}
\Ob [\langle \mcC \rangle \diamond \mcEinj] & = & \add (\Ob (\langle \mcC \rangle) \star \mcEInj ) \\
& = & \add (\add (\mcC) \star \mcEInj) \\
& \supseteq & \add (\mcC \star \mcEInj) = \Ob (\mcI^{\perp}).
\end{eqnarray*}
The first equality follows from Theorem~\ref{object ext}; the second from Proposition~\ref{add closure}. By Proposition~\ref{object syzygy}, $\mcI^{\perp}$ is an object ideal, so that $\langle \mcC \rangle \diamond \mcEinj = \mcI^{\perp}.$ By Theorem~\ref{syzygy ideal}, the ideal 
$\langle \mcC \rangle = \gro (\mcI)$ is then an $\mcI$-syzygy ideal, as required.  
\end{proof}

\section{The Ghost Lemma} 

This section is devoted to the study of special $\mcI \mcJ$-precovers, in case $\mcI$ and $\mcJ$ are special precovering ideals. So let $A$ be an object of $\mcA,$ and suppose that there exists a special $\mcI$-precover $i_1 : C_1^{\mcI} \to A$ of $A$ that appears as part of the ME-conflation
in $\Arr (\mcA)$ given by
\DIAGV{80}
{\xi_{\mcI} :} \n {\gro_{\mcI}} \n {\ear} \n {c^{\mcI}} \n {\Ear{i}} \n {1_A,}
\diag 
where $c^{\mcI} : C_0^{\mcI} \to C_1^{\mcI}$ and $C_1^{\mcI}$ has a special $\mcJ$-precover $j'_1 : C_1^{\mcJ} \to C_1^{\mcI}$ that arises as part of the ME-conflation of arrows
\DIAGV{80}
{\xi'_{\mcJ} :} \n {\gro_{\mcJ}} \n {\ear} \n {c^{\mcJ}} \n {\Ear{j'}} \n {1_{C_1^{\mcI}}.}
\diag 
Compose the ME-conflation $\xi'_{\mcJ}$ with the pullback along the morphism given by the arrow \linebreak 
$c^{\mcI} : C_0^{\mcI} \to C_1^{\mcI}$ to obtain the ME-conflation
\DIAGV{80}
{W_0} \n {\ear} \n {C} \n {\ear} \n {C_0^{\mcI}} \nn
{\seql} \n {} \n {\saR{c}} \n {} \n {\saR{c^{\mcI}}} \nn
{W_0} \n {\ear} \n {C_0^{\mcJ}} \n {\ear} \n {C_1^{\mcI}} \nn
{\saR{\gro_{\mcJ}}} \n {} \n {\saR{c^{\mcJ}}} \n {} \n {\seql} \nn
{W_1} \n {\ear} \n {C_1^{\mcJ}} \n {\Ear{j'_1}} \n {C_1^{\mcI},} 
\diag
which will be called $\xi_{\mcJ}.$ If we further denote $c^{\mcJ}c$ by $c^{\mcI \mcJ},$ we may express this as the ME-conflation
\DIAGV{80}
{\xi_{\mcJ} :} \n {\gro_{\mcJ}} \n {\ear} \n {c^{\mcI \mcJ}} \n {\Ear{j}} \n {c^{\mcI}.}
\diag 
It is important to observe that $j_1 = j'_1 \in \mcJ.$ By Theorem~\ref{mono-epi}, a commutative diagram 
\DIAGV{80}
{\gro_{\mcJ}} \n {\eeql} \n {\gro_{\mcJ}} \nn
{\sar} \n {} \n {\sar} \nn
{\gro_{\mcJ} \star \gro_{\mcI}} \n {\ear} \n {c^{\mcI \mcJ}} \n {\Ear{ij}} \n {1_A} \nn
{\sar} \n {} \n {\saR{j}} \n {} \n {\seql} \nn
{\gro_{\mcI}} \n {\ear} \n {c^{\mcI}} \n {\Ear{i}} \n {1_A,} 
\diag
arises in $\Arr (\mcA),$ all of whose rows and columns are ME-conflations, by Axiom $E_1^{\op}$ for an exact category. Now 
$(ij)_1 = i_1 j_1 \in \mcI \mcJ$ and Theorem~\ref{mult/ext} implies that $\gro_{\mcJ} \star \gro_{\mcI} \in (\mcI \mcJ)^{\perp}.$ It follows that the ME-conflation in the middle row yields a special $\mcI \mcJ$-precover $i_1j_1: C_1^{\mcI \mcJ} \to A$ of $A.$ If the notation above is amended slightly, the equation $\gro_{\mcI \mcJ} = \gro_{\mcJ} \star \gro_{\mcI}$ suggests that the relationship between the domain of a special $\mcI$-precover of $A$ and its $\mcI$-syzygy is analogous to the relationship, expressed by the Chain Rule, between a differentiable function and its differential.

\begin{theorem} \label{syzygy ext} {\rm (The Chain Rule)}
Let $\mcI$ and $\mcJ$ be ideals and $A \in \mcA.$ If $i_1 : C^{\mcI}(A) \to A$ is an $\mcI$-special precover with $\mcI$-syzygy
$\gro_{\mcI} (A)$ and $j_1 : C^{\mcJ}(C^{\mcI}(A)) \to C^{\mcI}(A)$ is a $\mcJ$-special precover of $C^{\mcI}(A)$ with $\mcJ$-syzygy 
$\gro_{\mcJ} (C^{\mcI}(A)),$ then
$i_1j_1 : C^{\mcJ}(C^{\mcI}(A)) \to A$ is an $(\mcI \mcJ)$-special precover of $A$ with $(\mcI \mcJ)$-syzygy 
$$\gro_{\mcI \mcJ} (A) = \gro_{\mcJ}(C^{\mcI}(A)) \star \gro_{\mcI}(A).$$
If the precovers $i_1$ and $j_1$ are object-special, with kernels $\grO_{\mcI}(A)$ and $\grO_{\mcJ}(C^{\mcI}(A)),$ respectively, then
$$\grO_{\mcI \mcJ} (A) = \grO_{\mcJ}(C^{\mcI}(A)) \star \grO_{\mcI}(A).$$
\end{theorem} 

\begin{proof}
All that needs to be verified is the last statement. If $i_1$ and $j_1$ are object-special precovers, then we may take the $\mcI$-syzygy 
$\gro_{\mcI}(A)$ and the $\mcJ$-syzygy $\gro_{\mcJ}(C^{\mcI}(A))$ to be isomorphisms. The extension 
$\gro_{\mcI \mcJ} (A) = \gro_{\mcJ}(C^{\mcI}(A)) \star \gro_{\mcI}(A)$ is then also an isomorphism. Furthermore, if $\grO_{\mcI}(A)$ and 
$\grO_{\mcJ} (C^{\mcI}(A))$ are the associated object syzygies, then the isomorphism $\gro_{\mcI \mcJ}(A)$ is isomorphic in the arrow category to the identity morphism on some extension of objects $\grO_{\mcJ}(C^{\mcI}(A)) \star \grO_{\mcI}(A).$
\end{proof}

The Chain Rule yields the following important property of special precovering ideals.
 
\begin{corollary} \label{two ideals}
If $\mcI$ and $\mcJ$ are special precovering ideals of the exact category $(\mcA; \mcE),$ then so is the product ideal $\mcI \mcJ,$ with 
$$\gro (\mcI \mcJ) = \gro (\mcJ) \diamond \gro (\mcI).$$
If $\mcI$ and $\mcJ$ are object-special precovering ideals then so is $\mcI \mcJ,$ and 
$\grO (\mcI \mcJ) = \grO (\mcJ) \star \grO (\mcI).$
\end{corollary}

Let $\mcI = \mcJ$ in Corollary~\ref{two ideals} and iterate the process finitely many times to see that every special (resp., object-special) precovering ideal $\mcI$ of $(\mcA; \mcE)$ gives rise to a filtration
$$\Hom = \mcI^0 \supseteq \mcI \supseteq \mcI^2 \supseteq \cdots \supseteq \mcI^n \supseteq \cdots$$ 
of special (resp., object-special) precovering ideals. 
If $\mcI$ is an object-special precovering ideal, and $\grO (\mcI)$ is an $\mcI$-syzygy subcategory, then Corollary~\ref{two ideals} implies that, for every $n > 0,$ an $\mcI^n$-syzygy subcategory is given by the category $\grO (\mcI^n) = \grO (\mcI)^{\star n},$ the $n$-fold extension of $\grO (\mcI).$ The objects $U$ of this category are those for which there exists a filtration, that is, a sequence
\DIAGV{80}
{0 = U_0} \n {\Ear{i_1}} \n {U_1} \n {\Ear{i_2}} \n {\cdots} \n {\Ear{i_n}} \n {U_n = U}
\diag
of inflations, of length $n,$ whose cokernels lie in $\grO (\mcI).$ An important special case of Corollary~\ref{two ideals} is when $\gro (\mcI) = \mcI^{\perp}$ and $\gro (\mcJ) = \mcJ^{\perp}.$

\begin{corollary} \label{ideal ghost lemma}
Let $\mcI$ and $\mcJ$ be special precovering ideals of an exact category $(\mcA; \mcE)$ that has enough injective morphisms. Then
$(\mcI \mcJ)^{\perp} = \mcJ^{\perp} \diamond \mcI^{\perp}.$
\end{corollary}
 
\begin{proof}
By the Chain Rule, $\mcJ^{\perp} \diamond \mcI^{\perp}$ is an $\mcI \mcJ$-syzygy ideal, so that 
$(\mcJ^{\perp} \diamond \mcI^{\perp}) \diamond \mcEinj = (\mcI \mcJ)^{\perp},$ by Theorem~\ref{syzygy ideal}. 
By Proposition~\ref{ideal associativity} and the fact that $\mcI^{\perp}$ is an $\mcI$-syzygy ideal, the left side of the equation is equal to 
$\mcJ^{\perp} \diamond (\mcI^{\perp} \diamond \mcEinj) = \mcJ^{\perp} \diamond \mcI^{\perp}.$
\end{proof}

Collecting the observations of the previous two corollaries and their duals provides the centerpiece of our paper.

\begin{theorem} \label{the ghost lemma} {\rm (The Ghost Lemma)}
Let $(\mcA; \mcE)$ be an exact category with enough injective morphism and enough projective morphisms. The class of special precovering (resp., preenveloping) ideals is closed under products $\mcI \mcJ$ and extensions $\mcI \diamond \mcJ.$ Moreover, the bijective correspondence $\mcI \mapsto \mcI^{\perp}$ satisfies
$$(\mcI \mcJ)^{\perp} = \mcJ^{\perp} \diamond \mcI^{\perp} \;\; \mbox{and} \;\; (\mcI \diamond \mcJ)^{\perp} = \mcJ^{\perp} \mcI^{\perp}.$$
\end{theorem}

\begin{proof}
By Corollary~\ref{two ideals}, special precovering ideals are closed under products. The hypothesis allows us to invoke Salce's Lemma (Theorem~\ref{Salce}) to prove that special preenveloping ideals are closed under extensions: if $\mcK_1$ and $\mcK_2$ are special preenveloping ideals, then $\mcK_1 = \mcJ^{\perp}$ and $\mcK_2 = \mcI^{\perp}$ for some special precovering ideals $\mcJ$ and $\mcI.$ By Corollary~\ref{two ideals}, the product ideal $\mcI \mcJ$ is itself a special precovering ideal, so Salce's Lemma implies that $(\mcI \mcJ)^{\perp}$ is a special preenveloping ideal. By Corollary~\ref{ideal ghost lemma}, $(\mcI \mcJ)^{\perp} = \mcJ^{\perp} \diamond \mcI^{\perp} = \mcK_1 \diamond \mcK_2.$ Because the hypothesis of the theorem is self-dual, it follows that the special precovering ideals are closed under extensions, while the special preenveloping ones are closed under products. The first equation comes from Corollary~\ref{ideal ghost lemma}, while the second is nothing more that its dual. 
\end{proof}

\section{The Phantom Ideal}

The phantom ideal $\grF$ in $\RMod$ is an object-special precovering ideal, with a $\grF$-syzygy subcategory given by the category 
$\grO (\grF) = \RPinj$ of pure injective left $R$-modules~\cite[\S 6]{FGHT}. By Corollary~\ref{two ideals}, every finite power $\grF^n$ of the phantom ideal is itself an object-special precovering ideal, with a $\grF^n$-syzygy subcategory given by $\grO (\grF^n) = (\RPinj)^{\star n}.$ This is the additive category of modules $U$ that possess a filtration of length $n$ with pure injective factors. Proposition~\ref{ext by objects} then implies that $\Ob [(\grF^n)^{\perp}] = \add [(\RPinj)^{\star n} \star \RInj].$

Recall from the Introduction that a ring $R$ is semiprimary if the Jacobson radical $J = J(R)$ is nilpotent, and $R/J$ is semisimple artinian. The least number $n$ for which $J^n = 0$ is called the {\em nilpotency index} of $J.$

\begin{theorem} \label{semiprimary}
If $R$ is a semiprimary ring with $J^n = 0,$ then $\grF^n = \langle \RProj \rangle.$
\end{theorem}

\begin{proof}
We will prove that $\Ob [(\grF^n)^{\perp}] = \RMod,$ by showing that every left $R$-module $M$ has a filtration of length $n$ whose factors are pure injective, and thus belongs to $(\RPinj)^{\star n}.$ The conclusion $\grF^n = \langle \RProj \rangle$ then follows. Indeed, consider the {\em Loewy series}
$$M \supseteq JM \supseteq J^2M \supseteq \cdots \supseteq J^{n-1}M \supseteq J^n M = 0.$$
Each of the factors is semisimple, and therefore pure injective. This follows from the observation that if $N$ is a semisimple $R$-module, then it may be considered as an $R/J$-module. As such, it is injective, and therefore pure injective. But the quotient map $R \to R/J$ is a ring epimorphism, so that the action of $R$ on $N$ yields a pure injective $R$-module, by~\cite[Thm 5.5.3]{P}. Another way to see that a semisimple module $M$ over a semiprimary ring is pure injective is to note that it is of finite endolength~\cite[Cor 4.4.24]{P}. 
\end{proof}

A ring $R$ is {\em  Quasi-Frobenius} (QF)~\cite{NY} if the category of projective left $R$-modules coincides with the category of left injective $R$-modules. It is well-known that this property is left-right symmetric and that every QF ring is semiprimary. If $R$ is a QF ring, then the stable category of $\RMod$ is obtained as the quotient category of $\RMod$ by the ideal $\langle \RProj \rangle$ of projective/injective modules. It is denoted by $R$-$\uMod$ and has the structure of a triangulated category. The phantom ideal $\grF$ in $\RMod$ contains $\langle \RProj \rangle$ and so induces an ideal, also denoted by $\grF,$ in the stable category. It is obvious that for a QF ring, the equation $\grF^n = \langle \RProj \rangle$ holds in the module category $\RMod$ if and only if $\grF^n = 0$ in the stable category. The following proposition characterizes the nilpotency index of the phantom ideal in the stable category of modules over a QF ring.

\begin{proposition} \label{QF index}
If $R$ is a QF ring, then $\grF^n = 0$ in the stable category $R$-$\uMod$ if and only if every left $R$-module is a direct summand of a module that possesses a filtration of length $n$ with pure injective factors.
\end{proposition}

\begin{proof}
The equation $\grF^n = 0$ holds in the stable category if and only if $\grF^n = \langle \RProj \rangle$ in the module category $\RMod,$ which is equivalent to $\RMod = \Ob [(\grF^n)^{\perp}] = \add [(\RPinj)^{\star n} \star \RInj],$ by the discussion above. It is enough therefore to show that for a QF ring, $(\RPinj)^{\star n} \star \RInj = (\RPinj)^{\star n}.$ But if a module $M$ belongs to $(\RPinj)^{\star n} \star \RInj$ then there is a filtration of $M$
$$M = M_0 \supseteq M_1 \supseteq \cdots \supseteq M_n \supseteq M_{n+1} = 0$$
of length $n+1,$ all of whose factors are pure injective, except the first $M_0/M_1,$ which is injective. The first factor $M_0/M_1$ is then projective, so that $M = M_0/M_1 \oplus M_1,$ and we obtain the filtration
$$M = M_0/M_1 \oplus M_1 \supseteq M_0/M_1 \oplus M_2 \supseteq \cdots \supseteq M_0/M_1 \oplus M_n \supseteq 0$$
of length $n,$ all of whose factors are pure injective.
\end{proof}

If $R$ is a QF ring, then the nilpotency index of the Jacobson radical is a strict upper bound for the nilpotency index of the phantom ideal in the stable module category.

\begin{theorem} \label{QF}
If $R$ is a QF ring with Jacobson radical $J,$ then $J^n = 0$ implies that $\grF^{n-1} = 0$ in the stable category $R$-$\uMod.$
\end{theorem} 

\begin{proof}
First note that if $J = 0,$ then the QF ring is semisimple artinian, and therefore that the stable category $R$-$\uMod$ is equivalent to the trivial category $\{ 0 \}.$ In that case $\grF^0 = \uHom = 0,$ so that $\grF^{n - 1} = 0.$ So suppose that $J$ is nonzero. By Proposition~\ref{QF index}, it suffices to verify that every left $R$-module $M$ has a filtration of length $n-1$ with pure injective factors. Over a QF ring, every left $R$-module $M$ admits a direct sum decomposition $M = E \oplus M'$ where $E$ is a projective/injective module and $M'$ has no projective/injective summands. The Loewy length of $M'$ is at most $n-1.$ For, the injective envelope of $M'$ is part of the short exact sequence
\DIAGV{80}
{0} \n {\ear} \n {M'} \n {\ear} \n {E(M')} \n {\Ear{p}} \n {\grO^{-1}(M')} \n {\ear} \n {0,}
\diag  
where the morphism $p : E(M') \to \grO^{-1}(M')$ is the projective cover of the cosyzygy of $M'.$ It follows that $M'$ is a small submodule of its injective envelope, and therefore, that $M' \subseteq JE(M'),$ and hence that $J^{n-1}M' = 0.$ Consider now the filtration
$M \supseteq JM' \supseteq J^2M' \supseteq \cdots \supseteq J^{n-1}M' = 0,$
of length at most $n-1.$ The first factor is the pure injective module $E \oplus M'/JM',$ while the others are semisimple.
\end{proof}

Because the nilpotency index of the Jacobson radical $J$ of a group algebra $k[G]$ is bounded by the $k$-dimension $|G|$ of the algebra, Theorem~\ref{QF} provides an upper bound on the nilpotency index of the stable phantom ideal that makes no reference to the ground field $k.$ 
This answers a question~\cite[Question 5.6.3]{BG1} of Benson and Gnacadja in the affirmative.

\begin{corollary} \label{group bound}
Let $G$ be a finite group and $k$ a field. If $\grF$ denotes the ideal of phantom morphisms in the stable category $k[G]$-$\uMod$ of modules over the group algebra $k[G],$ then $\grF^{|G| - 1} = 0.$
\end{corollary} 

The Jennings-Quillen Theorem~\cite[p.\ 87]{B book} may be used to obtain upper bounds for the nilpotency index of the Jacobson radical as in~\cite{CCM}.
For example, if $G$ is a regular $p$-group of rank $r,$ this provides a phantom number of $(p-1)r,$ but if $G$ is a cyclic $p$-group, then the nilpotency index of the Jacobson radical is $|G|,$ because the group algebra $k[G]$ is uniserial, and therefore of finite representation type. In the proof of Theorem~\ref{semiprimary}, the Jacobson radical $J$ may be replaced by any nilpotent ideal $K$ for which the ring $R/K$ is left pure semisimple, i.e., of left pure global dimension $0.$ If $G$ is a cyclic $p$-group, then $K = 0$ will work: the nilpotency index of $K$ is $1,$ so that the stable category $k[G]$-$\uMod$ is phantomless. \bigskip

A module $F$ belongs to the ideal $\grF$ provided that $\Tor_1^R (-,F) = 0$ or, equivalently, if it is flat. Denote by $\RFlat \subseteq \RMod$ the subcategory of flat modules. The object ideal $\langle \RFlat \rangle$ of morphisms that factor through a flat module is contained in $\grF$ and, because it is idempotent, we see that the filtration 
$$\Hom = \grF^0 \supseteq \grF \supseteq \grF^2 \supseteq \cdots \supseteq \grF^n \supseteq \cdots \supseteq \langle \RFlat \rangle$$
of finite powers of $\grF$ is bounded below by $\langle \RFlat \rangle.$ Recall that a module $C$ is {\em cotorsion} if $\Ext^1_R (F,C) = 0$ for every flat module $F,$ and that~\cite{BEE} every module $M$ has a flat cover \linebreak $f: F(M) \to M$ whose kernel, denoted by $\grO^{\flat}(M),$ is cotorsion. 
Denote by $\RCotor \subseteq \RMod$ the subcategory of cotorsion left $R$-modules.

\begin{proposition} \label{cotorsion syzygy}
If $\RCotor \subseteq \Ob [(\grF^n)^{\perp}],$ then $\grF^n = \langle \RFlat \rangle.$
\end{proposition}

\begin{proof}
Let $N$ be a module and consider the short exact sequence 
\DIAG
{0} \n {\ear} \n {\grO^{\flat} (N)} \n {\ear} \n {F(N)} \n {\Ear{f}} \n {N} \n {\ear} \n {0,} 
\diag
where $f : F(N) \to N$ is the flat cover of $N,$ and the module $\grO^{\flat} (N)$ is cotorsion. The hypothesis implies that 
$\grO^{\flat}(N) \in \Ob [(\grF^n)^{\perp}].$ Because the morphism $f$ belongs to $\grF^n,$ it is an object-special $\grF^n$-precover of $N:$
every morphism in $\grF^n$ with codomain $N$ factors through $F(N)$ and so belongs to $\langle \RFlat \rangle.$
\end{proof}

The remainder of this section is devoted to developing a criterion sufficient for the condition $\grF^{n+1} = \langle \RFlat \rangle$ to hold in $\RMod$ for a right coherent ring $R.$ 

\begin{lemma}
If the ring $R$ is right coherent, then $\add [(\RPinj)^{\star n}]$ is invariant under $\grO^{\flat}.$
\end{lemma}

\begin{proof}
Let us prove that every module $U$ in $(\RPinj)^{\star n}$ has a flat syzygy, not necessarily $\grO^{\flat} (U),$ that also belongs to
$(\RPinj)^{\star n}.$ Because the flat cover of a finite direct sum of modules is the direct sum of the respective flat 
covers~\cite[\S 1.4]{X}, this will imply that if $M$ belongs to $\add [(\RPinj)^{\star n}],$ then so does the kernel $\grO^{\flat} (M)$ of its flat cover. The proof proceeds by induction on $n.$ The case $n=1$ is the statement that the flat syzygy of a pure injective left $R$-module is itself pure injective, a result proved by Xu~\cite[Lemma 3.2.4]{X}.

If $U \in (\RPinj)^{\star (n+1)},$ then there is a short exact sequence, shown at the bottom of the commutative diagram
\DIAGV{80}
{} \n {} \n {0} \n {} \n {0} \n {} \n {0} \n {} \n {} \nn
{} \n {} \n {\sar} \n {} \n {\sar} \n {} \n {\sar} \n {} \n {} \nn
{0} \n {\ear} \n {\grO^{\flat} (U_0)} \n {\ear} \n {K} \n {\ear} \n {\grO^{\flat} (U_n)} \n {\ear} \n {0} \nn
{} \n {} \n {\sar} \n {} \n {\sar} \n {} \n {\sar} \n {} \n {} \nn
{0} \n {\ear} \n {F(U_0)} \n {\ear} \n {F(U_0) \oplus F(U_n)} \n {\ear} \n {F(U_n)} \n {\ear} \n {0} \nn
{} \n {} \n {\sar} \n {} \n {\sar} \n {} \n {\sar} \n {} \n {} \nn
{0} \n {\ear} \n {U_0} \n {\ear} \n {U} \n {\ear} \n {U_n} \n {\ear} \n {0} \nn
{} \n {} \n {\sar} \n {} \n {\sar} \n {} \n {\sar} \n {} \n {} \nn
{} \n {} \n {0} \n {} \n {0} \n {} \n {0,} \n {} \n {} 
\diag
where $U_0$ is pure injective, $U_n$ belongs to $(\RPinj)^{\star n},$ and all the rows and columns are exact. The left and right columns are given by the flat covers of $U_0$ and $U_n,$ respectively. Because $U_0$ is pure injective, it is cotorsion, so that the flat cover of $U_n$ lifts to $U,$ which yields, as in the Horseshoe Lemma~\cite[Lemma 2.5.1]{B book}, a flat precover of $U$ in the middle column. By the case $n=1,$ the flat syzygy $\grO^{\flat} (U_0)$ is pure injective. By the induction hypothesis, the flat syzygy $\grO^{\flat} (U_n)$ belongs to $(\RPinj)^{\star n}.$ Therefore, 
$K$ belongs to $(\RPinj)^{\star (n+1)}.$ 
\end{proof}

\begin{theorem} \label{pi cores}
Suppose that $R$ is right coherent and let $C$ be a cotorsion left $R$-module with a coresolution
\DIAGV{60}
{0} \n {\ear} \n {C} \n {\Ear{c}} \n {I_0} \n {\ear} \n {I_1} \n {\ear} \n {\cdots} \n {\ear} \n {I_n} \n {\ear} \n {0}
\diag
in $\RMod$ with each $I_k$ pure injective. Then $C \in \add [(\RPinj)^{\star (n+1)}].$
\end{theorem}

\begin{proof}
The proof proceeds by induction on $n.$ The case $n=0$ is a tautology. To prove the induction step, consider the commutative diagram
\DIAGV{80}
{} \n {} \n {} \n {} \n {0} \n {} \n {0} \n {} \n {} \nn
{} \n {} \n {} \n {} \n {\sar} \n {} \n {\sar} \n {} \n {} \nn
{} \n {} \n {} \n {} \n {\grO^{\flat}(C')} \n {\eeql} \n {\grO^{\flat}(C')} \n {} \n {} \nn
{} \n {} \n {} \n {} \n {\sar} \n {} \n {\sar} \n {} \n {} \nn
{0} \n {\ear} \n {C} \n {\ear} \n {U} \n {\ear} \n {F(C')} \n {\ear} \n {0} \nn
{} \n {} \n {\seql} \n {} \n {\sar} \n {} \n {\sar} \n {} \n {} \nn
{0} \n {\ear} \n {C} \n {\Ear{c}} \n {I_0} \n {\ear} \n {C'} \n {\ear} \n {0} \nn
{} \n {} \n {} \n {} \n {\sar} \n {} \n {\sar} \n {} \n {} \nn
{} \n {} \n {} \n {} \n {0} \n {} \n {0,} \n {} \n {} 
\diag 
where the the bottom row is the short exact sequence that begins the given coresolution, and the rest of the diagram is obtained by pullback along the cokernel of $c$ and the flat cover of $C'.$ Both $C$ and $I_0$ are cotorsion, so that $C'$ is also a cotorsion module with a coresolution by pure injective modules of properly shorter length. The induction hypothesis therefore applies and we may assume that $C'$ belongs to 
$\add [(\RPinj)^{\star n}].$ By the previous lemma, so does the flat syzygy $\grO^{\flat} (C').$ Because $C$ is cotorsion, the flat cover of $C'$ 
lifts to $I_0$ and causes the middle row of the diagram to split. It follows that $C$ is a direct summand of $U.$ 

Now $\grO^{\flat} (C')$ belongs to $\add [(\RPinj)^{\star n}],$ so there exists a module $K$ such that 
$\grO^{\flat}(C') \oplus K \in (\RPinj)^{\star n}.$ The module $U \oplus K$ is an extension of the pure injective module $I_0$ by 
$\grO^{\flat}(C') \oplus K,$ so that $U \oplus K$ belongs to $(\RPinj)^{\star (n+1)}.$ Consequently, $C \in \add [(\RPinj)^{\star (n+1)}].$
\end{proof}

Theorem~\ref{pi cores} and Proposition~\ref{cotorsion syzygy} imply the following.

\begin{corollary} \label{Benson-G} 
Let $R$ be a right coherent ring such that every cotorsion left $R$-module $C$ has a coresolution
\DIAGV{60}
{0} \n {\ear} \n {C} \n {\ear} \n {I_0} \n {\ear} \n {I_1} \n {\ear} \n {\cdots} \n {\ear} \n {I_n} \n {\ear} \n {0}
\diag
with each $I_k$ pure injective. Then $\grF^{n+1} = \langle \RFlat \rangle.$
\end{corollary}

A ring $R$ is said to be of left {\em pure global dimension} at most $n$ if every left $R$-module has a pure injective coresolution of length at most $n.$ Such a ring clearly satisfies the hypothesis of Corollary~\ref{Benson-G}. This yields the following generalization of a result of Benson and Gnacadja~\cite{BG1}, which asserts that for a group algebra $k[G]$ of pure global dimension $n$ the finite power $\grF^{n+1}$ of the phantom ideal is the object ideal of morphisms that factor through a projective left $R$-module.

\begin{corollary}
If $R$ is a right coherent ring of left pure global dimension at most $n,$ then $\grF^{n+1} = \langle \RFlat \rangle.$ 
\end{corollary}

Similarly, if every left $R$-module has an injective coresolution of length at most $n,$ then the hypothesis of Corollary~\ref{Benson-G} is satisfied.

\begin{corollary}
If $R$ is a right coherent ring of homological dimension at most $n,$ then $\grF^{n+1} = \langle \RFlat \rangle.$ 
\end{corollary}

A ring $R$ is of {\em flat global dimension} at most $n$ if every left $R$-module has a flat resolution of length at most $n.$ Then every cotorsion left $R$-module has injective dimension at most $n,$ so that the hypothesis of Theorem~\ref{pi cores} is satisfied and $\grF^{n+1} = \langle \RFlat \rangle.$ To see why every left cotorsion module $C$ has injective dimension at most $n,$ consider a flat resolution of $F_{\ast} \to M$
\DIAGV{80}
{0} \n {\ear} \n {F_n} \n {\ear} \n {\cdots} \n {\ear} \n {F_1} \n {\ear} \n {F_0} \n {\ear} \n {M} \n {\ear} \n {0} 
\diag 
of length $n,$ of an arbitrary left $R$-module $M.$ This resolution is $\Hom (-,C)$-acyclic~\cite[Prop III.1.2A]{Hart} so that $\Ext^k(M,C)$ is given by the homology of $\Hom (F_{\ast}, C)$ at $\Hom (F_k,C).$ In particular, $\Ext^{n+1}(M,C) = 0.$ Since this is true for every left $R$-module $M,$ it follows from standard homological arguments that $C$ has a coresolution by injective modules of length at most $n,$ and the hypothesis of Corollary~\ref{Benson-G} is again satified.

\begin{corollary}
If $R$ is a right coherent ring of flat global dimension at most $n,$ then $\grF^{n+1} = \langle \RFlat \rangle.$ 
\end{corollary}

For example, if a ring $R$ is right semihereditary, then it is right coherent and of flat global dimension at most $1$ so that 
$\grF^2 = \langle \RFlat \rangle.$

\section{Wakamatsu's Lemma}

The Ghost Lemma (Theorem~\ref{the ghost lemma}) implies that in the bijective correspondence $\mcI \mapsto \mcI^{\perp}$ given by Salce's Lemma (Theorem~\ref{Salce}), as well as its inverse $\mcK \mapsto {^{\perp}}\mcK,$ idempotent ideals $\mcI^2 = \mcI$ correspond to ideals closed under extensions $\mcK \diamond \mcK = \mcK.$ These two properties of an ideal are familiar from the classical theory, because is $(\mcF, \mcC)$ is a complete cotorsion pair, then both ideals in the complete ideal cotorsion pair $(\langle \mcF \rangle, \langle \mcC \rangle)$ (see~\cite[Thm 28]{FGHT}) are idempotent and closed under extensions. They are idempotent, because they are object ideals; they are closed under extensions, because the underlying subcategories of objects are. In this section, we take up the study of ideals having these properties, but not with the usual assumption that they be special precovering, but, rather, covering. None of the results in this section require the existence of enough injective or projective morphisms. \bigskip  

Let $\mcI$ be an ideal of an exact category $(\mcA; \mcE)$ and $A$ an object of $\mcA.$ An $\mcI$-precover $i : C \to A$ of $A$ is an $\mcI${\em -cover} if every endomorphism $f : C \to C$ that makes the diagram
\DIAGV{80}
{} \n {} \n {C} \nn
{} \n {\Swar{f}} \n {\saR{i}} \nn
{C} \n {\Ear{i}} \n {A} 
\diag 
commute is an automorphism. Recall that an $\mcI$-precover is necessarily a deflation. The notion of an $\mcI${\em -envelope} is defined dually. In what follows, we state the results in terms of covers, rather than envelopes, leaving the formulations and proofs of the dual results to the interested reader.

If $i : C \to A$ is an $\mcI$-cover and $i' : C' \to A$ is an $\mcI$-precover, then there is a morphism from $c : C \to C'$ over $A$ that induces a morphism $k$ on the kernels, and therefore an ME-conflation of arrows
\DIAGV{80}
{\Xi:} \n {K} \n {\ear} \n {C} \n {\Ear{i}} \n  {A} \nn
{\Sar{\xi}} \n {\saR{k}} \n {} \n {\saR{c}} \n {} \n  {\seql} \nn
{\Xi':} \n {K'} \n {\ear} \n {C'} \n {\Ear{i'}} \n  {A.} 
\diag 
The condition that $i$ be an $\mcI$-cover implies the existence of a retraction $\sigma : \Xi' \to \Xi$ of $\xi,$ which implies that the conflation $\Xi$ is a direct summand of $\Xi'.$ In particular, both of the morphisms $c : C \to C'$ and $k : K \to K'$ have retractions. This will be used in the proof of following lemma, which is the main result of this section.

\begin{lemma} \label{Wakamatsu}
Let $\mcI$ be an ideal, closed under extensions, in an exact category $(\mcA; \mcE).$ If $A \in \mcA$ and $i : C \to A$ is the $\mcI$-cover of $A,$ then the kernel $K$ of $i$ belongs to $\Ob (\mcI^{\perp}).$ 
\end{lemma}

\begin{proof}
It must be shown that $\Ext (i', K) = 0$ for every morphism $i' : I'_0 \to I'_1$ in $\mcI.$ Equivalently, every ME-conflation 
$\xi : 1_K \stackrel{m}{\rightarrow} a \to i'$ is null homotopic. This is depicted by the diagram
\DIAGV{80}
{K} \n {\Ear{m_0}} \n {A_0} \n {\ear} \n {I'_0} \nn 
{\saR{1_K}} \n {} \n {\saR{a}} \n {} \n {\saR{i'}} \nn 
{K} \n {\ear} \n {A_1} \n {\ear} \n {I'_1.} 
\diag
We will use the hypothesis that $i : I \to A$ is an $\mcI$-cover to prove that 
$m_0 : K \to A_0$ is a split inflation. Let us take the pushout of $\xi$ along the ME-conflation $1_K \to 1_C \stackrel{i}{\rightarrow} 1_A$ to obtain the diagram
\DIAGV{80}
{1_K} \n {\ear} \n {1_C} \n {\Ear{i}} \n {1_A} \nn 
{\saR{m}} \n {} \n {\sar} \n {} \n {\seql} \nn 
{a} \n {\ear} \n {b} \n {\Ear{s}} \n {1_A} \nn
{\sar} \n {} \n {\saR{j}} \nn 
{i'} \n {\eeql} \n {i'}  
\diag
in $\Arr (\mcA ).$ By Theorem~\ref{mono-epi}, all the rows and columns of this diagram are ME-conflations. Regarded as a diagram in $\mcA,$ it is given by
\DIAGV{60}
{} \n {} \n {K} \n {} \n {\earv{100}} \n {} \n {C} \n {} \n {\Earv{i}{100}} \n {} \n {A} \nn
{} \n {\sweql} \n {} \n {} \n {} \n {\sweql} \n {} \n {} \n {} \n {\sweql} \n {} \nn
{K} \n {} \n {\cross{\earv{100}}{\saRv{\movename{m_0}{0}{-10}}{100}}} \n {} \n {C} \n {} \n {\cross{\earv{100}}{\sarv{100}}} \n {} \n {A} \n {} \n {\seqlv{100}} \nn
{} \nn
{\sarv{100}} \n {} \n {A_0} \n {} \n {\cross{\earv{100}}{\sarv{100}}} \n {} \n {B_0} \n {} \n {\cross{\Earv{\movename{s_0}{20}{0}}{100}}{\seqlv{100}}} \n {} \n {A} \nn
{} \n {\swaR{a}} \n {} \n {} \n {} \n {\swaR{b}} \n {} \n {} \n {} \n {\sweql} \n {} \nn
{A_1} \n {} \n {\cross{\earv{100}}{\sarv{100}}} \n {} \n {B_1} \n {} \n {\cross{\Earv{\movename{s_1}{20}{0}}{100}}{\saRv{\movename{j_0}{0}{-10}}{100}}} \n {} \n {A} \n {} \n {} \nn
{} \n {} \n {} \n {} \n {} \n {} \n {} \n {} \n {} \n {} \n {} \nn
{\sarv{100}} \n {} \n {I'_0} \n {} \n {\cross{\eeqlv{100}}{\saRv{\movename{j_1}{0}{-10}}{100}}} \n {} \n {I'_0} \nn
{} \n {\swaR{i'}} \n {} \n {} \n {} \n {\swaR{i'}} \n {} \n {} \n {} \n {} \n {} \nn
{I'_1} \n {} \n {\eeqlv{100}} \n {} \n {I'_1.} \n {} \n {} \n {} \n {} \n {} \n {}
\diag
If we can show that $s_0 : B_0 \to A$ belongs to $\mcI,$ then the properties of the $\mcI$-cover $i : C \to A$ will ensure that inflation the
$m_0 : K \to A_0$ has a retraction that yields a homotopy of $\xi.$ Let us factor $s_0$ as $s_0 = s_1b$ and extract from the above diagram the commutative diagram
\DIAGV{80}
{} \n {} \n {B_0} \n {} \n {} \nn
{} \n {} \n {\saR{b}} \n {\Sear{i'j_0}} \n {} \nn
{C} \n {\ear} \n {B_1} \n {\ear} \n {I'_1} \nn
{} \n {\seaR{i}} \n {\saR{s_1}} \n {} \n {} \nn
{} \n {} \n {A,}
\diag
where the composition of the middle column is given by $s_0 = s_1b$ and the middle row is the conflation that appears in the front middle column of the diagram above. Now $i$ and $i'j_0$ both belong to $\mcI,$ which is closed under extensions, so that Lemma~\ref{ext crit} implies that 
$s_0 \in \mcI,$ as required.  
\end{proof}

\begin{definition} 
An ideal $\mcI$ is {\em covering} if every object $A$ in $\mcA$ has an $\mcI$-cover.
\end{definition}

If $\mcI$ is a covering ideal in an exact category $(\mcA; \mcE),$ it is not known, even for the phantom ideal $\grF$ in $\RMod,$ if $\mcI^2$ is a covering ideal. The next result, whose proof is immediate from the previous lemma, is an ideal version of Wakamatsu's Lemma~\cite{W}.

\begin{theorem} {\rm (Wakamatsu's Lemma)}
Every covering ideal $\mcI,$ closed under extensions, is an object-special precovering ideal.
\end{theorem}

\begin{example} \label{ph ideals}
The phantom ideal $\grF$ of $\RMod$ is covering by~\cite[Thm 7]{H}. That it is closed under extensions follows from the fact that it is right $\Tor$-orthogonal to the category of all right $R$-modules. Precisely, let $h$ be a morphism in the extension ideal $\grF \diamond \grF.$ By Lemma~\ref{ext crit}, the morphism $h$ may be expressed as a composition $h = h^2 h^1$ as shown in the commutative diagram
\DIAGV{80}
{} \n {} \n {} \n {} \n {A} \n {} \n {} \nn
{} \n {} \n {} \n {} \n {\saR{h^1}} \n {\Sear{\grf_2}} \n {} \nn
{0} \n {\ear} \n {X} \n {\Ear{m}} \n {Z} \n {\Ear{p}} \n {Y} \n {\ear} \n {0} \nn
{} \n {} \n {} \n {\seaR{\grf_1}} \n {\saR{h^2}} \n {} \n {} \nn
{} \n {} \n {} \n {} \n {B,} \n {} \n {}
\diag
where $\grf_1$ and $\grf_2$ are phantom morphisms and the middle row is exact. Let $N = N_R$ be a right $R$-module and apply the covariant functor $\Tor_1 (N,-)$ to the diagram above to obtain the commutative diagram
\DIAGV{70}
{} \n {} \n {} \n {} \n {\Tor (N,A)} \nn
{} \nn
{} \n {} \n {} \n {} \n {\Sarv{\Tor (N,h^1)}{110}} \n {} \n {\Searv{0}{120}} \nn
{} \nn
{\Tor (N,X)} \n {} \n {\Earv{\Tor (N,m)}{80}} \n {} \n {\Tor (N,Z)} \n {} \n {\Earv{\Tor (N,p)}{80}} \n {} \n {\Tor (N,Y)} \nn
{} \nn
{} \n {} \n {\seaRv{0}{120}} \n {} \n {\saRv{\Tor (N,h^2)}{110}} \nn
{} \nn
{} \n {} \n {} \n {} \n {\Tor (N,B)}
\diag
of abelian groups, whose middle row is exact. As in the conclusion of the proof of Theorem~\ref{mult/ext}, $\Tor (N,h) = \Tor (N,h^2) \Tor (N,h^1) = 0,$ for every $N_R,$ which implies that $h$ is a phantom morphism. One concludes from Wakamatsu's Lemma the (known) fact that $\grF$ is an object-special precovering ideal.
\end{example}

Let us now turn our attention to idempotent ideals. If $\mcI$ and $\mcJ$ are ideals, every morphism in the product ideal is of the form $f = \sum_k i_kj_k  : A \to B,$ where each $j_k : A \to X_k$ belongs to $\mcJ$ and each $i_k : X_k \to B$ belongs to $\mcI.$ If $\mcI$ is precovering, then every $i_k$ factors through an $\mcI$-precover $i_B : X \to B,$ $i_k = i_Bg_k,$ where $g_k : X_k \to X.$ The morphism $f = \sum_k i_B g_kj_k = i_B (\sum_k g_kj_k)$ is therefore expressible as a composition of a morphism $i_B$ in $\mcI$ and a morphism in $\mcJ.$ The next result is a kind of dual to Wakamatsu's Lemma, because its subject is the covering idempotent ideals, rather than covering ideals closed under extensions.

\begin{proposition} \label{idem cover}
An idempotent covering ideal is an object ideal.
\end{proposition}
 
\begin{proof}
Let $\mcI$ be an idempotent covering ideal and suppose that $i : C \to A$ is an $\mcI$-cover of an object $A.$ The foregoing comments imply that we may express $i$ as a composition $i = i_1 i_2$ of morphisms in $\mcI,$
\DIAGV{80}
{C} \n {\Ear{i}} \n {A} \nn 
{\saR{i_2}} \n {} \n {\seql} \nn 
{B} \n {\Ear{i_1}} \n {A} \nn
{\saR{f}} \n {} \n {\seql} \nn 
{C} \n {\Ear{i}} \n {A.}
\diag
Because $i_1 : B \to A$ belongs to $\mcI,$ it will factor as shown above, $i_1 = if.$ Because $i : C \to A$ is an $\mcI$-cover, the endomorphism 
$fi_2 : C \to C$ is an invertible morphism in $\mcI.$ It follows that $1_C \in \mcI,$ and therefore that every morphism in $\mcI$ with codomain $A$ factors through the object $C \in \Ob (\mcI).$ 
\end{proof} 

A ring $R$ is called {\em phantomless} if the phantom ideal is an object ideal, $\grF = \langle \RFlat \rangle.$

\begin{proposition} \label{object}
A ring $R$ is phantomless if and only if the phantom ideal is idempotent. This is equivalent to $\RCotor \subseteq \Ob (\grF^{\perp}).$
\end{proposition}

\begin{proof}
The first equivalence follows from the fact~\cite[Thm 7]{H} that $\grF$ is a covering ideal and Proposition~\ref{idem cover}. The second statement follows from Proposition~\ref{cotorsion syzygy} and the definition of a cotorsion module.
\end{proof}

\begin{proposition} \label{idem 2}
A special (resp., object-special) precovering ideal $\mcI$ is idempotent if and only if some $\mcI$-syzygy ideal (resp., subcategory) is closed under extensions. 
\end{proposition}

\begin{proof}
If $\mcI$ is idempotent, then $\gro (\mcI) = \mcI^{\perp}$ is closed under extensions, by Corollary~\ref{idem 1}. Conversely, suppose that some $\mcI$-syzygy ideal $\gro (\mcI)$ is closed under extensions. By Corollary~\ref{two ideals}, 
$\gro (\mcI^2) = \gro (\mcI) \diamond \gro (\mcI) = \gro (\mcI).$ Let $A$ be an object of $\mcA$ and consider a special $\mcI^2$-precover $i_1 : C_1 \to A$ of $A$ with $\mcI^2$-syzygy $\gro : W_0 \to W_1$ in $\gro (\mcI^2)$ as shown in
\DIAGV{70}
{W_0} \n {\ear} \n {C_0} \n {\ear} \n {A} \nn
{\saR{\gro}} \n {} \n {\sar} \n {} \n {\seql} \nn
{W_1} \n {\ear} \n {C_1} \n {\Ear{i_1}} \n {A.} 
\diag
Then $i_1 \in \mcI^2 \subseteq \mcI$ and $\gro \in \gro (\mcI^2) \subseteq \gro (\mcI),$ so that $i_1 : C_1 \to A$ is a special $\mcI$-precover. It follows that every morphism in $\mcI$ with codomain $A$ factors through $i_1$ and therefore belongs to $\mcI^2.$ Thus 
$\mcI \subseteq \mcI^2.$

If $\mcI$ is an idempotent object-special precovering ideal, then $\mcI^{\perp}$ is closed under extensions. The $\mcI$-syzygy subcategory
$\grO (\mcI) = \Ob (\mcI^{\perp})$ is then closed under extensions, because, as in the proof of Theorem~\ref{object ext},
$$\Ob (\mcI^{\perp}) \star \Ob (\mcI^{\perp}) \subseteq \Ob (\mcI^{\perp} \diamond \mcI^{\perp}) = \Ob (\mcI^{\perp}).$$ 
Suppose, on the other hand, that some $\mcI$-syzygy subcategory $\grO (\mcI)$ is closed under extensions, $\grO (\mcI) \star \grO (\mcI) \subseteq \grO (\mcI).$ By Corollary~\ref{two ideals}, the subcategory $\grO (\mcI) \star \grO (\mcI) = \grO (\mcI^2)$ is an $\mcI^2$-syzygy subcategory. If $A$ is an object of $\mcA,$ and $i : C \to A$ is an object-special $\mcI^2$-precover 
\DIAGV{80}
{\grO_{\mcI^2}(A)} \n {\ear} \n {C} \n {\Ear{i}} \n {A,}
\diag
whose kernel lies in $\grO (\mcI^2) \subseteq \grO (\mcI),$ then, because $i \in \mcI,$ the morphism is an object-special $\mcI$-precover. As above, $\mcI \subseteq \mcI^2$ and $\mcI$ is idempotent.  
\end{proof}

Every pure injective module is cotorsion and the subcategory of cotorsion modules is closed under extension. Proposition~\ref{ext by objects} thus yields the inclusions $\RPinj \subseteq \Ob (\grF^{\perp}) \subseteq \RCotor.$ A ring $R$ is a called a {\em Xu} ring if the equality $\RPinj = \RCotor$ holds. In that case, the $\grF$-syzygy subcategory $\RPinj$ is closed under extensions, so that Proposition~\ref{idem 2} implies that $\grF$ is idempotent and therefore that the ring $R$ is phantomless. Xu rings have been characterized in~\cite[Thm 3.5.1]{X} as follows. We offer a proof using the present theory. 
 
\begin{proposition} {\rm (Xu)}
Let $R$ be an associative ring with identity. Every cotorsion left $R$-module is pure injective if and only if the subcategory $\RPinj$ of pure injective left $R$-modules is closed under extensions.
\end{proposition}

\begin{proof}
If every cotorsion module is pure injective, then it is immediate that the subcategory $\RPinj$ is closed under extensions. Conversely, if $\RPinj$ is closed under extensions, then $R$ is phantomless and, $\Ob (\grF^{\perp}) = \add (\RPinj \star \RInj) = \RPinj,$ by Proposition~\ref{ext by objects}. By Proposition~\ref{object}, every cotorsion module belongs to $\Ob (\grF^{\perp}).$
\end{proof}

In the sequel to this article, we will develop a theory to prove the dual of this proposition.

\end{document}